\documentclass{article}
\usepackage{amsmath,amssymb,amsthm, amscd, verbatim, setspace}
\usepackage{graphicx} 
\usepackage{algorithm,algorithmic, color}
\usepackage[margin=1in]{geometry}
\usepackage{mathtools}
\usepackage[all]{xy}
\theoremstyle{plain}
\DeclareMathOperator{\Tor}{Tor}

\newtheorem{thm}{Theorem} 

\newtheorem*{thmn}{Theorem} 
\DeclareMathOperator{\cof}{-cof}
\DeclareMathOperator{\End}{End}

\DeclareMathOperator{\cell}{-cell}
\DeclareMathOperator{\inj}{-inj}
\numberwithin{thm}{section}

\newtheorem{lem}[thm]{Lemma}
\newtheorem{cor}[thm]{Corollary} 
\newtheorem{prop}[thm]{Proposition} 
\newtheorem{obs}[thm]{Observation}

\newtheorem{conv}[thm]{Convention}

\theoremstyle{definition}
\newtheorem{defn}[thm]{Definition}

\theoremstyle{remark}

\newtheorem{warn}[thm]{Warning}

\newtheorem{ex}[thm]{Example}

\DeclareMathOperator{\Ext}{Ext}

\DeclareMathOperator{\coker}{coker}

\DeclareMathOperator{\Hom}{Hom}

\DeclareMathOperator{\Gr}{Gr}
\DeclareMathOperator{\cyl}{cyl}
\DeclareMathOperator{\Mot}{Mot}
\DeclareMathOperator{\BP}{BP}
\DeclareMathOperator{\colim}{colim}
\DeclareMathOperator{\Mod}{Mod}

\DeclareMathOperator{\Fil}{Fil}

\DeclareMathOperator{\Sp}{Sp}
\newcommand{\heart}{\ensuremath\heartsuit}
\newcommand{\CC}{\mathbb{C}}

\newcommand{\ZZ}{\mathbb{Z}}

\newcommand{\FF}{\mathbb{F}}
\newcommand{\LL}{\mathbb{L}}
\newcommand{\cat}{\mathcal}
\newcommand{\sSet}{\text{sSet}}

\DeclareMathOperator{\Ch}{Ch}
\DeclareMathOperator{\Fun}{Fun}

\DeclareMathOperator{\im}{im}

\usepackage[T1]{fontenc}
\usepackage{subcaption,graphicx}
\usepackage{natbib}
\setcitestyle{numbers,open={[},close={]}}
\usepackage{mathtools}  
\usepackage{amssymb}    
\usepackage{amsthm}

\newcommand{\thesistitle}{Model Structures on Infinity-Categories of Filtrations}
\newcommand{\thesisauthor}{Colin Campbell Aitken}

\title{\thesistitle}

\author{\thesisauthor}

\usepackage{doi}
\usepackage{xurl}
\hypersetup{bookmarksnumbered,
            unicode,
            linktoc=all,
            pdftitle={\thesistitle},
            pdfauthor={\thesisauthor},
            pdfsubject={},                                
            pdfborder={0 0 0}}
\makeatletter
\let\ORG@hyper@linkstart\hyper@linkstart
\protected\def\hyper@linkstart#1#2{%
  \lowercase{\ORG@hyper@linkstart{#1}{#2}}}
\makeatother

\begin{document}
\maketitle

\abstract 
In 1974, Gugenheim and May showed that the cohomology $\Ext_A(R,R)$ of a connected augmented algebra over a field $R$ is generated by elements with $s = 1$ under matric Massey products. In particular, this applies to the $E_2$ page of the $H\FF_p$-based Adams spectral sequence. By studying a novel sequence of deformations of a presentably symmetric monoidal stable $\infty$-category $\cat C$, we show that for a variety of spectral sequences coming from filtered spectra, the set of elements on the $E_2$ page surviving to the $E_k$ page is generated under matric Massey products by elements with degree $s < k.$ This work is the author's PhD thesis, completed under the supervision of Peter May. 

\tableofcontents
 \newpage 

\begin{center}{\emph{This thesis is dedicated to the memory of Mary Aitken, who would have insisted on reading the whole thing.}}\end{center}

\section{Introduction}
Essentially all modern computations of the stable homotopy groups of spheres are based on some version of the Adams Spectral Sequence. Isaksen, Wang, and Xu \cite{isaksen2023stable} have recently used both classical and motivic forms of the spectral sequence to compute these homotopy groups through the 90-stem (up to a handful of specified uncertainties.) Similarly, Hill, Hopkins, and Ravenel's resolution \cite{hill2016nonexistence} of the Kervaire Invariant problem depends fundamentally on seminal work of Browder \cite{browder1969kervaire} reducing this fundamentally geometric question to a calculation in the $s = 2$ line of the Adams spectral sequence. 

While algorithms to compute the additive structure of the $E_2$ page \cite{bruner1993ext} and even the $E_3$ page \cite{ nassau2010secondary, chua2022} of the Adams Spectral Sequence are tractable on modern computers, higher differentials still pose a serious computational challenge. This difficulty is typified by a recent paper of Wang and Xu \cite{wang2017triviality}, which devotes nearly eighty pages to computing a single $d_3$ differential.

Differentials are typically computed using a grab bag of ad hoc tools, including comparisons to other spectral sequences (particularly synthetic and $\CC$-motivic Adams Spectral Sequences) and tricks invoking multiplicative structure. As a particularly simple example of the latter, an early paper of John Wang \cite{wang1967cohomology} resolves the Hopf Invariant One problem using little more than the multiplicative structure of the Adams Spectral Sequence and the fact that $h_0h_i^3$ is nonzero. More recent papers often make sophisticated use of \emph{Massey products} and their relationship with Toda brackets in the homotopy groups of spheres.

To avoid discussing signs in the introduction, we work over a field of characteristic two. Given elements $a,b,$ and $c$ of a differential graded algebra with $ab = bc = 0$ in homology, the Massey product $\langle a,b,c \rangle$ is the set of all homology classes $ae + fc$ with $d(e) = bc$ and $d(f) = ab.$ Moss \cite[Theorem~1]{moss} shows that if in addition $a d_k(b) = b d_k(c) = 0$,  we have a ``Leibniz rule''
\[
d_k \langle a,b,c \rangle \subseteq \langle d_k(a),b,c \rangle + \langle a,d_k(b),c \rangle + \langle a,b,d_k(c) \rangle
\]

May's article \cite{mmp} extends this idea to define \emph{matric Massey products}, which produce cohomology classes from more complicated sets of relations than $ab = bc = 0$, and showed that these higher products satisfy their own versions of the Leibniz rule. More generally, Kochman \cite[Theorem 8.2]{kochman1978chain} showed that matric Massey products on one page can be used to compute differentials on future pages, although this requires strong vanishing conditions that rarely hold in practice.

There are several limitations to this approach. The most significant, which we will not discuss further, comes from indeterminacies: the set
\[
\langle d_k(a),b,c \rangle + \langle a,d_k(b),c \rangle + \langle a,b,d_k(c) \rangle
\]
may be much larger than $d_k \langle a,b,c \rangle,$ in which case the Leibniz rule will not uniquely determine the differential. In this case one must either consider different Massey products or use another technique to find the value of the differential.

This paper addresses a more structural concern: to apply a Leibniz rule, we need to show that a given element can be written as a matric Massey product of ``smaller" pieces. If we consider ordinary products, this is not guaranteed even when $s > 1$: for example, when $p = 2$, the element $c_0$ with $s = 8, t = 11$ is indecomposable.

Gugenheim and May \cite{gm} show that this problem goes away if we consider matric Massey products, at least for the $E_2$ page of the classical ($H\FF_p$-based) Adams spectral sequence: this page is completely generated by the $s = 1$ line under matric Massey products. For example, while $c_0$ is irreducible by ordinary products, Bruner \cite[Figure 2.3]{bruner2009adams} shows that we have
\[
c_0 = \langle h_0, h_2^2, h_1 \rangle
\]

However, very few of the elements with $s= 1$ survive to the actual homotopy groups of spheres: when $p = 2$, only $h_0, h_1, h_2,$ and $h_3$ survive (detecting the elements $2, \eta, \nu,$ and $\sigma$ respectively), while for odd primes only classes detecting $p$ and $\alpha$ survive. So it is perhaps surprising that Cohen \cite{cohen} has shown that the $p$-complete homotopy groups of spheres are generated by elements detected by  $s = 1$ under matric Toda brackets. To the best of the author's knowledge, no further results along these lines are known.

This leaves a suggestive gap in the literature. The first well-defined page of the Adams spectral sequence is generated (in the above sense) by its $s = 1$ line, and so is the homotopy ring it converges to. Moreover, some Massey products survive to Toda Brackets in this homotopy ring \cite[Theorem~2]{moss}, although the presence of ``crossing differentials'' complicates this relationship. So while it is challenging to directly access the $E_k$ page of the spectral sequence for $2 < k < \infty$, it is not unreasonable to expect that these ``intermediate" pages should also be be generated by elements with small grading.

The most obvious generalization would be to ask that the $E_k$ page also be generated by the $s = 1$ line under matric Massey products. This is not the case: while $h_4^2$ is a permanent cycle, degree considerations show that it is not a matric Massey product on the $E_k$ page for $k \geq 4.$ Indeed, there are simply not enough differentials on the $E_k$ page for $k \geq 2$ to expect matric Massey products to generate much of these pages.

Instead, in this thesis, we look at the $E_{2,k}$ page of the classical Adams spectral sequence, by which we mean the subset of the $E_2$ page that survives to $E_k$. We show that this is generated by elements of degree $s < k$ for all $2 \leq k < \infty.$ To this end, we adapt one of the most spectacular advances in stable homotopy theory over the past decade: Gheorghe, Isaksen, Wang, and Xu's ``cofiber of tau'' philosophy. This philosophy stems from their observation that the stable motivic $\infty$-category $\Mot_\CC$ can be seen as a deformation whose generic fiber is the $\infty$-category of spectra and whose special fiber contains information about the $E_2$ page of the Adams Spectral Sequence.

More concretely, there is an element $\tau$ called the ``Tate twist'' in the (motivic, $p$-completed) stable stem $\pi_{0,-1}(S)$ such that:
\begin{enumerate}
\item The $\infty$-category $\Mod_{C\tau}(\Mot_{\CC})$ of modules over the cofiber of $\tau$ (``the special fiber $\tau = 0$") is equivalent to the stable $\infty$-category of $\BP_*\BP$ comodules concentrated in even degree.  \cite[Theorem~1.1]{gheorghe2021special}
    \item Given a(n ordinary p-complete) spectrum $X$, the motivic homotopy groups of $X \otimes C\tau$ recover the $E_2$ page of the BP-homology based Adams Spectral Sequence for $X.$ If $X$ is a commutative ring spectrum, this equivalence preserves higher multiplicative structure (e.g. Massey products).
\cite[Lemma~4.56]{pstrkagowski2023synthetic}
    \item The $\infty$-category $\Mot_{\CC}[\tau^{-1}]$ (``the generic fiber $\tau = 1$") is equivalent to the ordinary $\infty$-category of spectra. 
    \cite[Remark~1.15]{gheorghe2021special},\cite{dugger2010motivic}.
\end{enumerate}
Recent work by Isaksen, Wang, and Xu \cite{isaksen2023stable} exploits this deformation structure to relate the Adams spectral sequence to the algebraic Novikov spectral sequence, which is much more computable.  

For our purposes, property (2) is the most relevant: the entire higher multiplicative structure of the $E_2$ page of the Adams Spectral Sequence for $X$ is contained in the homotopy groups of $X \otimes C\tau$, which we can reason about using ordinary tools of stable homotopy theory. We will create a similar deformation to access the $E_k$ page of a more general spectral sequence, which will let us study matric Massey products on this page without having to keep track of all the information stored in the previous pages.

The cofiber of tau is also a key input into Pstragowski's theory of \emph{synthetic spectra} \cite{pstrkagowski2023synthetic}, which produces a similar deformation with BP replaced by a more general Adams-type homology theory $E$. Despite their novelty, synthetic spectra already have a variety of applications: they have been used by Burklund, Hahn, and Senger to compute Toda brackets in the stable homotopy category \cite{burklund2019boundaries}, by Patchkoria and Pstragowski to show that the homotopy category $E(n)$-local spectra is equivalent to an algebraic category for large primes \cite{patchkoria2021adams}, and by  Burklund to resolve an extension in the 54-stem of the homotopy groups of spheres \cite{burklund2021extension}.

In this thesis, we construct a  deformation whose special fiber corresponds with the $E_k$ page of the more general filtration spectral sequence, and show that this deformation carries enough multiplicative structure to support Gugenheim and May's proof nearly word-for-word. 

More concretely, for each positive integer $k$ and any sufficiently structured stable $\infty$-category $\cat C$, we consider the $\infty$-category $\Fil(\cat C) = \Fun(\ZZ, \cat C)$, where we consider $\ZZ$ to be a partially ordered set with one map $i \to j$ whenever  $i \geq j$. 

This has a bigraded suspension functor $\Sigma^{s,q}$ with $\Sigma^{s,q}X(n) = \Sigma^qX(n - s)$ and a natural map
\[
\tau: \Sigma^{-1,0}X \to X
\]
induced by the maps $X(n+1) \to X(n)$ defining $X$. 

We construct a localization of $\Fil(\cat C)$ we call $\cat D^k(\cat C)$ such that:
\begin{thmn}
Under reasonable conditions on $\cat C$, the following are true.
\begin{enumerate}
\item The cofiber of $\tau: S^{-1,0} \to S^{0,0}$ in $\cat D^k(\cat C)$ is equivalent to the cofiber of $\tau^{k+1}: S^{-k-1,0} \to S^{0,0}$ in $\Fil(\cat C)$. To avoid confusion, we refer to this object (which comes with a natural commutative ring structure) as $C\tau^{k+1}.$
\item The $\infty$-category of modules over $C\tau^{k+1}$ $($``the special fiber $\tau = 0$"$)$ is symmetric monoidally equivalent to the derived category of graded elements of $\cat C^{\heart}$, where the grading takes values in $\ZZ \times \{0,1,\ldots, k\}.$
\item For any $X$ in $\Fil(\cat C)$, the (bigraded) homotopy groups of $X \otimes C\tau^{k+1}$ in $\cat D^k(\cat C)$ recover the $E^{k+1}$ page of the filtration spectral sequence associated to $X$. If $X$ is a ring object, this equivalence preserves higher multiplicative information (e.g. matric Massey products).
\item The localization $\cat D^k(\cat C)[\tau^{-1}]$ $($``the generic fiber $\tau = 1$"$)$ is equivalent to $\cat C.$
\end{enumerate}
\end{thmn}
With this deformation in hand, we will adapt Gugenheim and May's proof to conclude: 
\begin{thmn}
Let $E_r^{s,t}$ be a multiplicative spectral sequence over a field $R$ concentrated in degrees $s \geq 0$, and suppose the $E_1$ page can be chosen to be freely generated by elements in degree $s = 1$.  Let $E_{2,r}^{s,t}$ denote the set of elements in $E_2^{s,t}$ which survive to the $E_r$ page. Then:
\begin{itemize}
\item The $E_2$ page is generated under matric Massey products by elements in degree $s = 1$. 
\item $E_{2,r}^{*,*}$ is generated under matric Massey products by elements in degree $0 < s < r.$
\end{itemize}
\end{thmn}
The first part is a generalization of the main result of \cite[Chapter 5]{gm}, while the second is wholly new. The result holds in our motivating example of the Adams Spectral Sequence, but also more generally. In fact, we prove that this theorem holds as long as the first page is \emph{Koszul}: that is, generated by elements of degree $s = 1$, with relations in degree $s = 2$, with relations between relations in degree $s = 3$, and so on.

Along the way, we show that our deformation lets us immediately generalize Moss's convergence theorem. This theorem, originally stated and proved for the Adams Spectral Sequence, relates two kinds of higher multiplicative structure: Massey products, which are defined in any differential graded algebra, and Toda brackets, which are defined in on the homotopy groups of ring spectra. We state and prove a generalization of this result to matric Massey products in arbitrary multiplicative filtration spectral sequences in Section \ref{mosssection}, which subsumes all previously published versions we are aware of. The proof relies on lifting Massey products in the $E^k$ page to Toda brackets in $\cat D^k(\cat C),$ following Burklund's proof for the Adams Spectral Sequence using synthetic spectra\cite{burklundsynthetic}.

To construct these localizations, we will use Mazel-Gee's language of \emph{model $\infty$-categories} \cite{mg}. As in the classical theory of model categories, Mazel-Gee considers an $\infty$-category $\cat M$ with collections of 1-morphisms called cofibrations, fibrations, and weak equivalences satisfying certain lifting axioms, and shows that this structure can be used to construct the localization of $\cat M$ with respect to the chosen weak equivalences. 

\begin{thmn}
There is a cofibrantly generated model structure on the $\infty$-category $\Fil(\cat C)$ such that:
\begin{enumerate}
    \item The weak equivalences are maps that become isomorphisms on levelwise homotopy groups after multiplying by $\tau^k.$
\item Every object is fibrant, and a necessary condition for an object to be cofibrant is that the $E_1$ through $E_{k-1}$ pages of its filtration spectral sequence are identical and degreewise projective.\end{enumerate}
\end{thmn}
In particular, weak equivalences induce isomorphisms on the $E_k$ page of the objects' filtration spectral sequences, although the converse is not true. We will define $\cat D^k(\cat C)$ to be the homotopy $\infty$-category of this model structure.

\subsection{Roadmap}
This thesis is organized into five sections. 

In Section \ref{filtersection}, we review the definition and basic properties of the $\infty$-category $\Fil(\cat C)$ of filtered objects of a presentably symmetric monoidal stable $\infty$-category $\cat C$. The material in this section is not new, but provides language and background for the rest of the thesis.

In Section \ref{modelsection}, we build cofibrantly generated model structures on $\Fil(\cat C)$, which we call ``the $k$-projective model structure'' for each nonnegative integer $k$. These model structures are closely related to the map $\tau$: in particular, the weak equivalences are the maps which are isomorphisms on homotopy groups after multiplying by $\tau^k$, and the generating cofibrations and acyclic cofibrations are built from spheres and the cofiber of $\tau^k.$ We start by reviewing Mazel-Gee's theory of model $\infty$-categories in Section \ref{modelreview}. In Section \ref{modelconstruct}, we fix for each $k$ sets $I^k$ and $J^k$ of morphism in $\Fil(\cat C)$ and prove that they generate a model structure. In Section \ref{symmonsection} we promote these to symmetric monoidal model structures, and in Section \ref{modelhomspaces} we show how to compute hom-spaces in the homotopy $\infty$-category $\cat D^k(\cat C)$ of the $k$-projective model structure. In Section \ref{compactspheres} we check that these localizations are compactly generated and stable, a technical condition we will need in later chapters. 

In Section \ref{sseqsection}, we review the construction of the spectral sequence of a filtered object $X$, with multiplicative structure if $X$ is a ring object. The main results of this section are well-known, but we need to relate them to the various objects in $\cat D^k(\cat C)$ for the next two sections, and the easiest way to do that is to prove them from scratch. Central are Theorems \ref{sseqform} and \ref{sseqmult}, which show that the $E_k$ page (including its multiplicative structure) is captured by tensoring $X$ with the cofiber of $\tau^{k+1}.$ As examples, we show how to construct the Adams Spectral Sequence in our framework in Section \ref{adamssection}, and similarly construct various Ext and Tor spectral sequences in Section \ref{exttorsec}. 

Section \ref{mmpsection} is the devoted to understanding the multiplicative structure of the various localizations $\cat D^k(\cat C)$ we've produced. In Section \ref{deformsec}, we build a t-structure on $\cat D^k(\cat C)$ and use it to show that we can view $\cat D^k(\cat C)$ as a deformation whose generic fiber is $\cat C$ and whose special fiber is an algebraic category capturing the $E^k$ page of the filtration spectral sequence. In Section \ref{realmmpsection} we use Ariotta's work on coherent chain complexes to define an $\infty$-categorical variant of matric Massey products, which we use in Section \ref{mosssection} to prove a stronger version of Moss's Convergence Theorem than has previously appeared in the literature. 

Finally, in Section \ref{gmsection}, we reinterpret Gugenheim and May's result on using matric Massey products to generate the cohomology of algebras as a statement about Koszulity conditions. The first two sections provide information and som classical background. In Section \ref{gmkoszulsection} we generalize their proof to show that any augmented differential graded algebra whose underlying algebra is Koszul (for example, freely generated by elements in degree $s = 1$) is generated by elements in degree $n = -1$ under matric Massey products. In Section \ref{filkoszulsection} we generalize further from differential graded algebras to multiplicative spectral sequences over a field whose $E^1$ page can be chosen to be Koszul, and prove that the set of elements of the $E^2$ page that survive to the $E^k$ page is generated in degrees $1 \leq s \leq k-1$ under matric Massey products. 

In two appendices we provide proofs of short lemmas and propositions we use in the main text, but whose proofs are not particularly enlightening.
\subsection{Notation and Conventions}
Throughout this paper, we use the language of stable $\infty$-categories freely. We follow the usual convention of phrasing $\infty$-categorical statements model-independently whenever possible, although for concreteness we implicitly use the theory of quasi-categories invented by Boardman and Vogt \cite{boardman1973homotopy} and further developed by Joyal \cite{joyal2002quasi, joyal2008notes} and Lurie \cite{htt, ha}. A ``commutative diagram'' in an $\infty$ category will  always mean ``commutative up to coherent homotopy'', and we will freely say ``\emph{the} object X'' or call $X$ unique as long as it is uniquely determined up to a contractible set of choices.

By ``ring'' (resp. ``commutative ring'') we mean an object of a given stable symmetric monoidal $\infty$-category equipped with an $A_\infty$ (resp. $E_\infty$) multiplication. Given a specified map $f: X \to Y$, we use the notations $Cf$ or $Y/X$ interchangeably to refer to the cofiber of $f$. We will use ``limit'' and ``colimit'' in the $\infty$-categorical sense, saving the words ``homotopy limit'' and ``homotopy colimit'' to refer to (co)limits in the localizations we build. 
\subsection{Acknowledgments}
I am grateful to the many friends who have made even the most painful parts of my time in grad school a genuine joy. It is not an exaggeration to say that I would not have made it through the past six-and-a-half years without you all. I love you dearly and thank you for the many late-night laughs, the genuine comfort, the thrilling competition, and the constant companionship. 

I am particularly grateful to my advisor Peter May for his kind support and patience, continuing to be there for me even as I found that my long-term research interests lay outside of topology. I am also thankful to the many other graduate students, postdocs, and professors who helped me find my place in the topological world, particularly including Akhil Mathew, XiaoLin Danny Shi, and Dylan Wilson. An extra thanks to those outside the University of Chicago who took the time to answer my questions related to this particular topic, including Shaul Barkan, Robert Burklund, Andy Senger, and Aaron Mazel-Gee. 

I am extremely grateful to my research collaborators for stimulating conversations and fascinating work, and welcoming a theoretical mathematician into applied work with open arms. I would especially like to thank Luiza Andrade, Joy Dada, Michael Kremer, Alex Lehner, Jeremy Lowe, Elisa Maffioli, Amy Pickering, Julie Powers, and Witold Więcek. 

I do not have words to sufficiently thank my family, who have given me so much and loved me so deeply over the past twenty-eight years. I love you all very much. Especially the nieces and nephews.
\section{Filtered objects in stable $\infty$-categories}
\label{filtersection}
\subsection{The $\infty$-category $\Fil(\cat C).$}
Throughout this thesis, we fix a presentably symmetric monoidal stable $\infty$-category $\cat C$. In this chapter, we review the definition and basic properties of the $\infty$-category $\Fil(\cat C)$ of filtered objects of $\cat C$. The material in this chapter is mostly well-known, and is provided primarily to review language and concepts we will need for the remainder of the thesis.

\begin{defn}
Given a cocomplete stable $\infty$-category $\cat C$, a \emph{filtered $\cat C$-object} is an object of the functor $\infty$-category
\[
\Fun(\mathbb{Z}, \cat C)
\]
where $\mathbb{Z}$ is an ordinary category with one arrow $i \to j$ whenever $i > j.$ We denote the $\infty$-category of filtered $\cat C$-objects by $\Fil(\cat C).$ One immediately sees that $\Fil(\cat C)$ is stable, and limits and colimits are computed point-wise. 
\end{defn}
For a filtered $\cat C$-object $X$, we will use the notation $X_n$ to mean $X(n).$ We can think of $\Fil(\cat C)$ as an extension of $\cat C$, as follows:
\begin{prop}\label{idhom}
The functor
\[
(-)_0:  \Fil(\cat C) \to \cat C
\]
sending $X$ to $X_0$ has a fully faithful left adjoint $\iota: \cat C \to \Fil(\cat C).$
\end{prop}
Intuitively, $\iota$ is the functor taking $X$ to the filtered object
\[
\cdots \to 0 \to 0 \to X \to X \to \cdots,
\]
with copies of $X$ in each non-positive degree connected by identity morphisms.

Our presentability assumption allows us to promote the symmetric monoidal structure on $\cat C$ to one on $\Fil(\cat C)$:
\begin{thm}\label{csymmon} 
The symmetric monoidal structure on $\cat C$ induces one on $\Fil(\cat C)$ via Day Convolution. This symmetric monoidal structure has the following properties:
\begin{itemize}
    \item Let $S^0$ denote the unit for the tensor product in $\cat C$. Then, the unit for the tensor product in $\Fil(\cat C)$ is $\iota S^0$.
    \item The tensor product on $\Fil(\cat C)$ preserves colimits in each variable.
    \item The functor $\iota$ is symmetric monoidal.
    \item $\Fil(\cat C)$ contains an internal function functor $F(-,-)$ such that $F(A,-)$ is right adjoint to the functor $A \otimes -.$
\end{itemize}
\end{thm}
\begin{proof}
See \cite{glasman2013day} and its extension in Section 2.2.6 of \cite{ha}. The particular case of filtered spectra is described in \cite{gheorghe2022c}.
\end{proof}
\begin{defn}
The bigraded suspension $\Sigma^{n,t}X$ is the functor sending $X_\bullet$ to $\Sigma^t X_{\bullet - n}.$ 
\end{defn}
Note that $\Sigma^{0,1}$ is the ``actual'' suspension functor coming from the stable $\infty$-category structure on $\Fil(\cat C)$. Later in this paper we will define model structures on $\Fil(\cat C)$ whose localizations have $\Sigma^{-n,1}$ as their corresponding suspension functors for various choices of $n.$

Filtrations come with a natural degree-shifting map, which will play the same role in our localizations that $\tau$ does in motivic or synthetic spectra. We keep the same name:
\begin{defn}\label{tau}
For any $X \in \Fil(\cat C)$, there is a map $\tau: \Sigma^{-1,0}X \to X$ induced by the maps $X_n \to X_{n-1}.$ 
\end{defn}
We also require a notion of ``spheres'' in $\cat C,$ which will be indexed by some abelian group $\cat A$ containing a distinguished copy of $\mathbb{Z}$ as a direct summand. We therefore assume there exists a symmetric monoidal functor $S^{(-)}$ from $\cat A$ (viewed as a discrete symmetric monoidal category) to $\cat C$ such that:
\begin{itemize}
    \item $S^{(n)} = \Sigma^n S^{(0)}.$
    \item The set $S^t$, as $t$ ranges across all of $\cat A$, is a set of compact generators for $\cat C.$
\end{itemize}
We then obtain ``suspension'' functors $\Sigma^t$ for any $t \in A$, which are necessarily equivalences with inverses $\Sigma^{-t}$.
\begin{defn}
Let $S^{n,t}$ denote $\Sigma^{n,0}\iota S^t$. Using Definition \ref{tau}, we obtain a map
\[
S^{-k,0} \xrightarrow{\tau^k} S^{0,0}
\] 
We denote by $C\tau^k$ the cofiber of this map. More generally, for any $X$ we have a map
\[
\Sigma^{-k,0} X \xrightarrow{\tau^k} X.
\]
We denote the cofiber of this map by $X \otimes C\tau^k.$ 
\end{defn}

\begin{warn}
We use the notation $C\tau$ to denote the cofiber of $\tau$ in $\Fil(\cat C).$ In the model structures we construct on $\Fil(\cat C)$, $\tau$ will not be a cofibration and so $C\tau$ will no longer be its cofiber after localization. Confusingly, in Corollary \ref{ctweird} we will see that the cofiber of $\tau$ in the $k$-projective model structure is actually $C\tau^{k+1}$.
\end{warn}
To maintain clarity, for a map $f: X \to Y$ we will always use the notation  $Cf$ or $Y/X$ to mean the corresponding cofiber \emph{in the $\infty$-category $\Fil(\cat C)$}. If we want to take the cofiber in some localization of $\Fil(\cat C)$, we will say so explicitly in words. 

Finally, we note the following:
\begin{thm}\label{ctcomm}
Suppose $\cat C$ satisfies the conditions of Theorem \ref{csymmon}. For each $k$, $C\tau^k$ is a commutative  ring in $\Fil(\cat C)$. The unit map comes from the cofiber sequence
\[
S^{-k,0} \to S^{0,0} \to C\tau^k.
\]
\end{thm}
\begin{proof}
This is a generalization of Theorem 3.2.5 of \cite{lurie2015rotation}, and the proof proceeds along similar lines. 

Let $[-k,0]$ denote the set of integers between $-k$ and $0$ (inclusive) and denote by $\Fil_{[-k,0]}(\cat C)$ the functor $\infty$-category $\Fun([-k,0], \cat C)$ and by $\Fil_{\leq 0}$ the functor $\infty$-category $\Fun(\ZZ_{\leq 0}, \cat C).$ We have a functor
\[
\iota_*: \Fil_{\leq 0}(\cat C) \to \Fil_{[-k,0]}(\cat C)
\]
by restricting the domain, which has a fully faithful right adjoint
\[
\iota^{!}: \Fil_{[-k,0]}(\cat C) \to \Fil_{\geq 0}(\cat C) 
\]
such that
\[
\iota_*(X_0 \to \cdots \to X_{-k}) = X_0 \to \cdots \to X_{-k} \to 0 \to 0 \to 0 \to \cdots 
\]
The usual Day Convolution formula implies that the symmetric monoidal structure  $\Fil_{\leq 0}(\cat C)$ inherits from $\Fil(\cat C)$ is compatible (in the sense of Definition 2.2.1.6 of \cite{ha}) with the localization coming from the adjunction $(\iota^*, \iota_*),$ so we obtain a symmetric monoidal structure on $\Fil_{[-k,0]}(\cat C)$ with unit $\iota^*(S^{0,0}).$ 

In particular, this implies that $\iota_*\iota^*(S^{0,0})$ is a commutative ring, whose unit is the unit map
\[
S^{0,0} \to \iota_*\iota^*(S^{0,0}),
\]
which simple inspection identifies with the cofiber map $S^{0,0} \to C\tau^k.$
\end{proof}
\begin{cor}\label{ctsarealgebras}  
Under the conditions of the previous theorem, we have a chain of commutative ring maps
\[
S^{0,0} \to \cdots \to C\tau^3 \to C\tau^2 \to C\tau.
\]
\end{cor}
\begin{proof}
Apply the same argument to the adjunctions coming from the chain of inclusions\[
[0,0] \subseteq [-1,0] \subseteq [-2,0] \subseteq \cdots \subseteq \mathbb{Z}_{\leq 0}
\]
\end{proof}

\section{Model structures on $\infty$-categories of filtrations}\label{modelsection}
In the previous chapter, we fixed a choice of ``spheres'' $S^t$ in $\cat C$. From this, we obtain several notions of homotopy groups of an object of $\Fil(\cat C)$:
\begin{defn}\label{homodef} For any $t \in A$, the \emph{$t$-th topological homotopy group}\footnote{Throughout this thesis, we use $t$ for total degree and $n$ for filtration degree.} of a (filtered) object $X$ is the abelian group
\[
\pi^{top}_t(X) = [S^t, \colim X_\bullet].
\]
The \emph{level-wise homotopy groups} of $X$ are given by
\[
\pi^{0}_{n,t}(X) = [S^{n,t}, X]_{\Fil(\cat C)} = [S^t, X_n]_{\cat C}.
\]
Notice that $\pi^{0}_{**}$ comes with a natural $\tau$ action, so we define $k$\emph{-sustained homotopy groups} of $X$ to be
\[
\pi^k_{n,t}(X) = \tau^k \pi_{n,t}^{0}(X) := \im\left(\pi_{t}(X_n) \xrightarrow{\tau^k} \pi_t(X_{n-k})\right)
\]
Equivalently,
\[\pi^{k}_{n,t}(X) = \coker(\pi_{t+1}( X_{n-k}/X_{n}) \to \pi_{t}(X_{n}))
\]
\end{defn}
In this chapter, we will use these invariants to build model $\infty$-category structures on $\Fil(\cat C).$
\subsection{Review of model $\infty$-categories}\label{modelreview}
This section is a review of Mazel-Gee's theory of model $\infty$-categories, which we use extensively throughout the rest of this thesis.  As is often the case in $\infty$-category theory, many of the results will not be surprising to readers who are comfortable with ordinary model (1-)categories. The proofs, however, are much harder, and we will not reproduce them here. 

A key goal of model ($\infty$-)category theory is to provide information about \emph{localizations} of ($\infty$-)categories at a collection of arrows:
\begin{defn}
Let $\cat M$ be an $\infty$-category, and $\cat W$ be a sub-$\infty$-category. The localization $\cat M[\cat W^{-1}],$ if it exists, is the unique $\infty$-category such that for any $\infty$-category $\cat C$,
\[
\Fun(\cat M[\cat W^{-1}], \cat C) = \Fun^{\cat W}(\cat M, \cat C),
\]
where the right hand side is the full sub-$\infty$-category of all functors $\cat M \to \cat C$ taking arrows in $\cat W$ to isomorphisms in $\cat C.$
\end{defn}

In general, set-theoretic problems can prevent the construction of localizations. Quillen introduced model structures as a way of both demonstrating the existence of localizations and facilitating computation with them. Mazel-Gee's definition in the $\infty$-categorical case closely mirrors Quillen's. 
\begin{defn}[\cite{mg}, Definition 1.1.1.]
A \emph{model $\infty$-category} consists of an $\infty$-category $\cat M$ with three wide subcategories $\cat W, \cat C, \cat F$ such that:
\begin{enumerate}
\item $\cat M$ is finitely bicomplete.
\item $\cat W$ satisfies the two-out-of-three property.
\item $\cat W, \cat C,$ and $\cat F$ are closed under retracts.
\item There exists a lift in any commutative square
\[
\xymatrix{
x \ar[r] \ar^i[d] & z \ar^p[d] \\
y \ar[r] \ar@{-->}[ru] & w
}
\]
in which $i$ is in $\cat C$, $p$ is in $\cat F$, and either $i$ or $p$ is in  $\cat W.$ 
\item Every arrow in $\cat M$ factors as an arrow in $\cat W \cap \cat C$ followed by an arrow in $\cat F$, and as an arrow in $\cat C$ followed by an arrow in $\cat W \cap \cat F.$
\end{enumerate}

The arrows in $\cat W$, $\cat C$, and $\cat F$ are called \emph{weak equivalences, cofibrations,} and \emph{fibrations}, respectively. Objects whose maps from the initial object are cofibrations are called \emph{cofibrant.} Dually, objects whose maps to the terminal object are fibrations are called \emph{fibrant.} Objects that are both cofibrant and fibrant are called \emph{bifibrant}.
\end{defn}
As in the classical case, we construct localizations by restricting to bifibrant objects and introducing a \emph{homotopy} relation. Classically, homotopy is an equivalence relation constructed by building a \emph{cylinder object} through which the map $X \coprod X \to X$ factors. The $\infty$-categorical homotopy construction is a bit subtler, since the higher homotopies require us to also consider all the maps $$X \coprod X \coprod \cdots \coprod X \to X$$ simultaneously. Mazel-Gee packages all this information into a cosimplicial object of $\cat M$ satisfying a few conditions:
\begin{defn} \cite[Definition~6.1.1]{mg}\label{cylinderdef}
Let $X$ be an object of a model $\infty$-category $\cat M.$ A \emph{cylinder object} for $X$ is a cosimplicial object $\cyl^\bullet(X)$ of $\cat M$ equipped with an equivalence $X \to \cyl^0(X)$ such that:
\begin{enumerate}
    \item The codegeneracy maps $\cyl^n(X) \xrightarrow{\sigma^i} \cyl^{n-1}(X)$ are all weak equivalences.
    \item The latching maps $L_n\cyl^\bullet(X) \to \cyl^n(X)$ are all cofibrations.
\end{enumerate}
\end{defn}
The homotopy relation is now elegantly realized as a geometric realization:
\begin{thm}\cite[Theorem~6.1.9]{mg}\label{homspaces}
Let $X$ be a cofibrant object in a model $\infty$-category $\cat M$, and let $Y$ be a fibrant object. For any cylinder object $\cyl^\bullet(X)$ of $X$, there is an equivalence
\[
|\Hom_{\cat M}(\cyl^\bullet(X),Y)| \to \Hom_{\cat \cat M[\cat W^{-1}]}(X,Y),
\]
where the left hand side is the geometric realization of the corresponding simplicial space.
\end{thm}
In general, it can be difficult to prove that the axioms of a model $\infty$-category hold. One class of model structures that are easier to construct are known as \emph{cofibrantly-generated} model $\infty$-categories.
\begin{defn}
A model $\infty$-category $\cat M$ is \emph{cofibrantly generated} if there exist sets $I$ and $J$ of homotopy classes of maps such that:
\begin{itemize}
\item Every map in $I$ is in $\cat C$
\item Every map in $J$ is in $\cat C \cap \cat W.$
\item The sources of maps in $I$ are small (in Quillen's sense) relative to maps in $I$.
\item The sources of maps in $J$ are small (in Quillen's sense) relative to maps in $J$.
\item $\cat F \cap \cat W$ is the collection of maps with the right lifting property with respect to $I$.
\item $\cat F$ is the collection of maps with the right lifting property with respect to $J$.
\end{itemize}
\end{defn}
\noindent Cofibrantly-generated model categories come with the following \emph{recognition theorem}:
\begin{thm}
Let $\cat M$ be an $\infty$-category with all colimits and finite limits, and let $\cat W$ be a subcategory of $\cat M$ which is closed under retracts and satisfies the two-out-of-three property. Suppose $I$ and $J$ are sets of homotopy classes of maps such that
\begin{itemize}
    \item The source of maps in $I$ (resp. $J$) are small relative to maps in $I$ (resp. $J$.)
    \item $J\cof \subseteq I\cof \cap \cat W$
    \item $I\inj \subseteq J\inj \cap \cat W$
    \item either $I\cof \cap W \subseteq J\cof$ or $J\inj \cap \cat W \subseteq I\inj.$
\end{itemize}
Then the sets $I$ and $J$ define a cofibrantly generated model structure on $\cat M$ whose weak equivalences are $\cat W.$
\end{thm}
In many applications, the underlying $\infty$-category of $\cat M$ has a monoidal structure that we would like to import to the localization $\cat M[\cat W^{-1}].$ This will follow from essentially the same conditions as the classical case.
\begin{defn}\cite[Definition 5.5.1]{mg}
A \emph{(symmetric)} monoidal structure on a model $\infty$-category of $\cat M$ consists of a (symmetric) monoidal structure on $\cat M$ such that:
\begin{enumerate}
    \item For any pair of cofibrations $f: X \to Y$ and $f': X' \to Y'$, the induced morphism
    \[
    (X \otimes Y') \coprod_{X \otimes X'} (Y \otimes X') \to Y \otimes Y'
    \]
    is a cofibration, and is a weak equivalence if either $f$ or $f'$ is.
    \item There exists a cofibrant replacement $\mathbb CS$ of the (symmetric) monoidal unit $S$ such that for any object $X$ the induced maps
    \[
    \mathbb{C}S \otimes X \to X
    \]
    and
    \[
    X \otimes \mathbb{C}S \to X
    \]
    are weak equivalences. 
\end{enumerate}
\end{defn}
\noindent As usual, we have:
\begin{thm}\cite[Propositions 5.5.4 and 5.5.6]{mg}
A (symmetric) monoidal structure on $\cat M$ gives rise to a canonical (symmetric) monoidal structure on $\cat M[\cat W^{-1}].$ For any two objects $x$ and $y$ of $\cat M$, their tensor product in $\cat M[\cat W^{-1}]$ can be identified with their tensor product in $\cat M$ as long both are cofibrant. 
\end{thm}
Finally, model structures are classically used to define \emph{derived functors} and in particular homotopy limits and colimits.

\begin{defn}
Let $\cat C$ be a small $\infty$-category and $\cat M$ be any model $\infty$-category. The \emph{projective model structure} on $\Fun(\cat C, \cat M)$, if it exists, is the model structure whose weak equivalences $\cat W_{proj}$ and fibrations $\cat F_{proj}$ are determined objectwise.
\end{defn}
\begin{defn}
Let $\cat C$ be a small $\infty$-category and $\cat M$ be any model $\infty$-category. The \emph{injective  model structure} on $\Fun(\cat C, \cat M)$, if it exists, is the model structure whose weak equivalences and cofibrations are determined objectwise.
\end{defn}
\begin{lem}
The projective model structure exists if $\cat M$ is cocomplete, cofibrantly generated, and the sources of its generating sets $I$ and $J$ are compact. The injective model structure exists if $\cat M$ is cofibrantly generated and presentable.
\end{lem}
\begin{proof}
See Remark 5.1.10 of \cite{mg}, the references therein, and the corresponding footnote. 
\end{proof}
\begin{defn}
Suppose $\cat C$ is a small $\infty$-category, $\cat M$ admits $\cat C$-shaped colimits, and the projective model structure exists on $\Fun(\cat C, \cat M).$ The adjunction
\[
\text{colim}: \Fun(\cat C,\cat M) \rightleftarrows \cat M: \text{const}
\]
descends to an adjunction
\[
\Fun(\cat C, \cat M)[\cat W^{-1}_{proj}] \rightleftarrows \cat M[\cat W^{-1}],
\]
whose left adjoint we call the \emph{homotopy colimit} functor. Similarly, if $\cat C$ is small, $\cat M$ admits $\cat C$-shaped limits, and the injective model structure exists on $\Fun(\cat C, \cat M)$, we obtain an adjunction
\[
\text{const}: \cat M \rightleftarrows \Fun(\cat C, \cat M): \text{lim}
\]
descending to an adjunction
\[
\cat M[\cat W^{-1}] \rightleftarrows \Fun(\cat C, \cat M)[\cat W^{-1}_{inj}]\]
whose right adjoint we call the \emph{homotopy limit functor.}
\end{defn}
\begin{warn}
The terminology here is a bit confusing, since colimits in $\infty$-categories are often called homotopy colimits. We will use the word ``colimit'' to refer to colimits in $\cat C$ or $\Fil(\cat C)$, and reserve ``homotopy colimit'' for this construction which (we will see shortly) generally compute colimits in a \emph{localization} of the ambient $\infty$-category.
\end{warn}
 Here we note an important difference between the $1$-categorical theory and the $\infty$-categorical theory. The homotopy category of a model category very rarely admits interesting limits or colimits. Instead, in many cases homotopy (co)limits and compute the (co)limits in the $\infty$-category produced by localizing \emph{in the $\infty$-categorical sense} at the weak equivalences, which is almost always distinct from the homotopy category.

In the theory of model $\infty$-categories, there is no such distinction: under minor assumptions, homotopy (co)limits in $\cat M$ compute (co)limits in $\cat M[\cat W^{-1}].$
\begin{thm}
Suppose $\cat C$ is small, $\cat M$ admits $\cat C$-shaped colimits, and the projective model structure exists for $\cat M$. Then, the homotopy colimit of a diagram $\cat C \to \cat M$ is equivalent to the colimit of the induced diagram $\cat C \to \cat M[\cat W^{-1}].$
\end{thm}
\begin{thm}
Suppose $\cat C$ is a small (ordinary) category, $\cat M$ admits $\cat C$-shaped limits, and the injective model structure exists for $\cat M$. Then, the homotopy limit of a diagram $\cat C \to \cat M$ is equivalent to the limit of the induced diagram $\cat C \to \cat M[\cat W^{-1}].$
\end{thm}
\begin{proof}
We prove the result for homotopy colimits. The existence of the projective model structure implies that $\cat M$ is hereditary in the sense of \cite[Definition~7.9.4]{cis} with respect to the standard universe, so \cite[Theorem~7.9.8]{cis} implies that we can identify $\Fun(\cat C, \cat M)[\cat W^{-1}_{proj}]$ with $\Fun(\cat C, \cat M[\cat W^{-1}]).$ Under this identification, the right adjoint in the homotopy colimit adjunction is the constant functor, whose left adjoint is the $\cat C$-shaped colimit functor in $\cat M[\cat W^{-1}]$. The proof for homotopy limits is essentially identical.
\end{proof}
\subsection{The $k$-projective model structure}\label{modelconstruct}
 In this section, we show that there is a model structure whose weak equivalences are the $\pi^k_{**}$-isomorphisms, which form a class $\cat W^k.$ This model structure can be seen as an $\infty$-categorification of the model structure defined on the category of filtered chain complexes by Cirici, Egas Santander, Livernet, and Whitehouse \cite{cirici2020model}, whose monoidal structure was explored by Brotherston\cite{brotherston2024monoidal}\footnote{The author apologizes that he is not able to provide a fuller exposition of either article, as he became aware of them after leaving the field of topology and has less time than he'd hoped to further edit this paper. He encourages anyone interested to read the originals. }
The author suspects the further model structures Brotherston explores in \cite{brotherston2024distributive} have interesting $\infty$-categorical analogues as well, but he has left the field of stable homotopy theory and does not expect to follow up on this idea. 
 \begin{defn}
 Fix a non-negative integer $k$. A map $f:X \to Y$ of filtered $\cat C$-objects is $k$-exact if the induced map
 \[
 \pi_{**}^k(X) \to \pi_{**}^k(Y)
 \]
 is an isomorphism. We denote the class of $k$-exact maps by $W^k.$
 \end{defn}

In this section we construct a cofibrantly-generated model structure on $\Ch(\cat C)$, analogous to the usual model structure on chain complexes. To do so, we define a set of generating cofibrations and generating acyclic cofibrations:
\begin{defn}
Let $I^k$ be the set of inclusions $S^{n,t} \to \Sigma^{n,t}C\tau^k$. Similarly, let $J^k$ be the set of inclusions $0 \to \Sigma^{n,t}C\tau^k$.
\end{defn}

\begin{prop}\label{inj}
We can characterize $I^k\inj$ and $J^k\inj$ as follows:
\begin{enumerate}
    \item The $J^k$ injections are the maps which are $\pi_t$-surjections on each $X_{n+k}/X_n$ for all $t.$
    \item The $I^k$ injections are maps with are $\pi_t$-surjections on each $X_{n+k}/X_n$ and such that $\pi_{t+1}(X_{n+k}/X_n) \to \pi_{t+1}(Y_{n+k}/Y_n) \times_{\pi_{t}(Y_n)} \pi_{t} X_n$ is a surjection for each $n$ and $t$.
\end{enumerate}
\end{prop}
\begin{proof}
Immediate from Proposition \ref{idhom}.
\end{proof}
To produce a cofibrantly generated model structure, we need to show that $I^k$ injections are the same as $k$-exact $J^k$ injections. The majority of this is contained in the next two lemmas:
\begin{lem}\label{part4}
Every $I^k$-injection is also $k$-exact.
\end{lem}
\begin{proof}
Fix an $I^k$-injection $f:X \to Y.$ By Proposition \ref{inj}, the map
\[
\pi_{t+1}(X_{n-k}/X_n) \to \pi_{t+1}(Y_{n-k}/Y_n) \times_{\pi_{t}(Y_n)} \pi_{t} X_n
\]
is surjective for each $n$ and $t$, so the map \[
\coker(\pi_{t+1}(X_{n-k}/X_n) \to \pi_{t}(X_n)) \to \coker( \pi_{t+1}(Y_{n-k}/Y_n)  \to \pi_{t}(Y_n))
\]
must be injective. 

Now, the pushout square
\[
\xymatrix{
S^{n+k,t-1} \ar[r] \ar[d] & 0\ar[d] \\
\Sigma^{n+k,t-1}C\tau^k\ar[r] & S^{n,t},
}
\]
shows that the map $0 \to S^{n,t}$ is in $I\cell$ and therefore must have the left lifting property with respect to $f$.  This implies that the map $\pi_t(X_n) \to \pi_t(Y_n)$ is surjective, so the map
\[
\coker(\pi_{t+1}(X_{n-k}/X_n) \to \pi_{t}(X_n)) \to \coker( \pi_{t+1}(Y_{n-k}/Y_n)  \to \pi_{t}(Y_n))
\]
must be as well. 
\end{proof}

\begin{lem}\label{part5}
Every $k$-exact $J^k$-injection is an $I^k$-injection.
\end{lem}
\begin{proof}
Let $f: X \to Y$ be a $k$-exact $J^k$-injection. By the characterization in Proposition \ref{inj}, we need to check that the map \[
\pi_{t+1}(X_{n-k}/X_n) \to \pi_{t+1}(Y_{n-k}/Y_n) \times_{\pi_{t}(Y_n)} \pi_{t}(X_n)
\]
is surjective for each $n$ and $t$. Let $F$ denote the fiber of $f$, and consider the diagram
\[
\xymatrix{
\pi_{t+1}(F_{n-k}/F_n) \ar[r]^{\delta_*} \ar[d]^{i_*} & \pi_t(F_n) \ar[d]^{i_*} \\
\pi_{t+1}(X_{n-k}/X_n) \ar[r]^{\delta_*} \ar[d]^{g_*} & \pi_t(X_n) \ar[d]^{n_*} \\
\pi_{t+1}(Y_{n-k}/Y_n) \ar[r]^{\delta_*} & \pi_t(Y_n)
}
\]
By Lemma \ref{acycfib}, we know that $F$ is $k$-acyclic, so the top arrow is surjective.  Now, suppose we have $x$ in $\pi_t(X_n)$ and $y$ in $\pi_{t+1}(Y_{n-k}/Y_n)$ such that $\delta_*(y) = f_*(x).$ Since $f$ is a $J^k$ injection, we can pick $z$ in $\pi_{t+1}(X_{n-k}/X_n)$ such that $f_*(z) = y.$ 

Then we have $f_*(\delta_*(z) - x) = 0$, so there exists $s$ in $\pi_t(F_n)$ with $i_*(s) = \delta_*(z) - x$. Since $F$ is $k$-acyclic, there exists $t$ in $\pi_{t+1}(F_{n-k}/F_n)$ with $\delta_*(t) = s$. Since $f_*(i_*(t)) = 0$ and $\delta_*i_*(t) = \delta_*(z) - x$, it follows that $z + i_*(t)$ is our desired preimage of both $x$ and $y$. 
\end{proof}

\begin{prop}
There is a cofibrantly-generated model $\infty$-category structure on $\Fil(\cat C)$ whose weak equivalences, generating cofibrations, and generating acyclic cofibrations are given by $W^k, I^k,$ and $J^k$ respectively.
\end{prop}
\begin{proof}
We check the conditions of \cite[Theorem~1.3.11]{mg} . $\Fil(\cat C)$ inherits cocompleteness and finite completeness from $\cat C$, while $W^k$ is closed under retracts and has the two-out-of-three property because \[
X \mapsto \pi^{k}_{n,t}(X)
\] is a functor. $J^k$ automatically permits the small object argument, while $I^k$ permits it by our compactness assumption on the collection of spheres. We therefore only need to check that:
\begin{enumerate}
    \item $J\cof \subseteq I\cof.$ 
    \item $J\cof \subseteq W.$
    \item $I\inj \subseteq J\inj$
    \item $I\inj \subseteq W$
    \item $(J\inj \cap W) \subseteq I\inj$
\end{enumerate}
(1) and (3) are immediate, because the generators of $J^k$ are all $I^k$-cell complexes. (4) and (5) are the statements of Lemmas \ref{part4} and \ref{part5}, respectively.

For (2), it suffices to check that $J\cell \subseteq W$, since $J\cof$ is the collection of retracts of maps in $J\cell$ and $W$ is closed under retracts. But maps in $J\cell$ take the form $X \to X \oplus P$, where $P$ is a coproduct of copies of $\Sigma^{n,t}C\tau^k$ for various $s$ and $q$, so the result follows because $\Sigma^{n,t}C\tau^k$ is $k$-exact.

\end{proof}
\begin{defn}
Let $\cat D^k(\cat C)$ denote the localization of $\Fil(\cat C)$ at $W^k$ using this model structure. 
\end{defn}
 Lemma \ref{repcomp} implies we have containments
 \[W^1 \subseteq W^2 \subseteq W^3 \subseteq \cdots,\]so we obtain successive localization functors
\[
\Fil(\cat C) \to \cat D^1(\cat C) \to \cat D^2(\cat C) \to \cdots.
\]
Unfortunately, the functors between the various $\cat D^k$s do not come from Quillen adjunctions (the cofibration inclusions point the wrong way), so our theory has little to say about them.

To gain some intuition, we will say a little more about the cofibrant objects of $\cat D^k(\cat C).$ 

\begin{thm}
Suppose $f:X \to Y$ is a cofibration in the $k$-projective model structure on $\Fil(\cat C)$. Then,
\begin{itemize}
    \item The induced map $X_n/X_{n+1} \to Y_n/Y_{n+1}$ is the inclusion of a direct summand, where the other direct summand $P_n$ is a retract of coproducts of $S^t$s for all $n$. 
    \item The map $P_n \to \Sigma Y_{n+1}$ lifts to $\Sigma Y_{n+k}$ for all $n$.
\end{itemize}
If $X_n = Y_n = 0$ for $n > 0$, the converse holds.
\end{thm}
\begin{proof}
Note that both properties are preserved by retracts, compositions, and pushouts, so it suffices to observe that that they hold for all generating cofibrations. 

We now prove the partial converse by exhibiting $f: X \to Y$ as a transfinite composition 
\[
X = X^{(1)} \to X^{(0)} \to X^{(-1)} \to \cdots
\]
so that each map is a cofibration, the colimit of the $X^{(i)}$ is $Y$, and $X^{(i)}$ agrees with $Y$ on $\mathbb{Z}_{\geq i}.$ Starting with $\ell = 0$, we assume that $P_\ell$ is a retract of coproducts of $S^t$s, and the map $C_\ell \to \Sigma X_{\ell + 1}$ lifts to $\Sigma X_{\ell + k}.$

Construct the pushout
\[
\xymatrix{
\Sigma^{\ell + k,-1}\iota P_\ell \ar[r] \ar[d] & \ar[d] \Sigma^{\ell + k,-1}\iota P_\ell \otimes  C\tau^k \\
X^{(\ell+1)} \ar[r] & X^{(\ell)}
}
\]
Since $C_\ell$ is a retract of coproduct of $S^t$s, the top map is a coproduct of retracts of generating cofibrations and therefore the map $X^{\ell + 1} \to X^{\ell}$ is a cofibration. Moreover, the commutative square
\[
\xymatrix{
\Sigma^{\ell + k,-1}\iota P_\ell \ar[r] \ar[d] & \ar[d] \Sigma^{\ell + k,-1}\iota P_\ell \otimes  C\tau^k \\
X^{(\ell+1)} \ar[r] & Y
}
\]
implies that we have compatible maps to $Y$. Since both $X$ and $Y$ are bounded above, we can use Theorem 4.7 of \cite{ariotta2021coherent} and consider the corresponding ``coherent chain complexes" supported in nonpositive degrees.

Note that the map
\[\Sigma^{\ell + k,-1}\iota C_\ell \to \Sigma^{\ell + k,-1}\iota C_\ell \otimes  C\tau^k \]
is an isomorphism outside of degree $-\ell$, so as a chain complex $X^{\ell}$ is supported in degrees between $0$ and $-\ell.$  Moreover, the induced map $X^{(\ell)} \to X$ in degree $\ell$ is just the identity on $C_\ell!$ In particular, $X^{(\ell)}$ agrees with $X$ on chain degree $\leq -\ell$, and therefore the same is true in filtration $\geq \ell$, as desired.
\end{proof}

\subsection{The $k$-projective model structure is closed symmetric monoidal}\label{symmonsection}

In this section, we show that our model structures extend to  symmetric monoidal model structures to build a compatible symmetric monoidal structure on the various $\cat D^k(\cat C)$s.

We begin this section by proving a lemma  about $C\tau^k$ and its modules that will come in handy later.
\begin{lem}\label{ctct}
For any $C\tau^k$-module $M$, there is a (non-canonical) equivalence
\[
M \otimes C\tau^k \simeq M \oplus \Sigma^{-k,1} M.
\]
\end{lem}
\begin{proof}
The cofiber sequence
\[
S^{-k,0} \xrightarrow{\tau^k} S^{0,0} \to C\tau^k
\]
gives a cofiber sequence
\[
\Sigma^{-k,0}M \xrightarrow{\tau^k} M \to M \otimes C\tau^k.
\]
Since $M$ is a $C\tau^k$-module, we must have $\tau^k = 0$ on $M$, so the cofiber sequence splits and we obtain a (non-canonical) identification
\[
M \otimes C\tau^k \simeq M \oplus \Sigma^{-k,1}M.
\]
\end{proof}

With Lemma \ref{ctct} in hand, the proof that the model structures we've constructed are closed symmetric monoidal is relatively straightforward. 
\begin{thm}
The $k$-projective model structure is closed symmetric monoidal
\end{thm}
\begin{proof}
Since $S^{0,0}$ is cofibrant, we only need to check the pushout-product axiom. By the argument of \cite[Corollary 4.2.5]{hovey2007model}, it suffices to check this on generating cofibrations and generating acyclic cofibrations. 

Further, since the generating cofibrations are all shifts of $S^{0,0} \to C\tau^k$ and the acyclic generating cofibrations are all shifts of $0 \to C\tau^k$, it suffices to consider the following two cases.

First, suppose $f$ and $f'$ are both the map $S^{0,0} \to C\tau^k.$ The pushout-product map takes the form
\[
\psi: C\tau^k \coprod_{S^{0,0}} C\tau^k \to C\tau^k \otimes C\tau^k.
\]
Using Lemma \ref{ctct}, note that this map fits into a pushout diagram
\[
\xymatrix{
S^{k,1} \ar[r] \ar[d] & C\tau^k \displaystyle \coprod_{S^{0,0}} C\tau^k \ar^\psi[d] \\
\Sigma^{k,1} C\tau^k \ar[r] & C\tau^k \otimes C\tau^k
}
\]
where the left vertical arrow is a cofibration, so $\psi$ is too. 

Next, suppose $f$ is as before, but $f'$ is the map $0 \to C\tau^k.$ We then have a map
\[
\varphi: C\tau^k \to C\tau^k \otimes C\tau^k
\]
which again using Lemma \ref{ctct} fits into a pushout square
\[
\xymatrix{
0\ar[r] \ar[d] & C\tau^k \ar^{\varphi}[d] \\
\Sigma^{k,1} C\tau^k \ar[r] & C\tau^k \otimes C\tau^k
}
\]
by which we immediately see that $\varphi$ is a cofibration, and acyclic because both its source and target are weak equivalent to $0.$ 
\end{proof}
This produces a canonical symmetric monoidal structure on $\cat D^k(\cat C)$, and we have the following.
\begin{cor}
The localization functor $\Fil(\cat C) \to \cat D^k(\cat C)$ is lax symmetric monoidal.
\end{cor}

\subsection{Mapping spaces in the derived $\infty$-category}\label{modelhomspaces}
In this section, we apply Theorem \ref{homspaces} to compute mapping spaces in $\cat D^k(\cat C).$
\begin{lem}
Let $\CC\tau^k$ denote the Amitsur complex of $C\tau^k$ 
\[
 C\tau^k \Rightarrow C\tau^k \otimes C\tau^k \Rrightarrow C\tau^k \otimes C\tau^k \otimes C\tau^k  \Rrightarrow \cdots
\]

The unit map gives a canonical map $\iota: S^{0,0} \to \CC\tau^k$, viewing the source as a constant cosimplicial object. If we define
\[
\Delta = \Sigma^{k,0} C\iota,
\]
then $\Delta$ is a cylinder object for $S^{0,0}$ in the $k$-projective model structure in the sense of Definition \ref{cylinderdef}.
\end{lem}
\begin{proof}
Using the equivalence in Lemma \ref{ctct}, we obtain an equivalence
\[
(C\tau^k)^{\otimes n+1} = \bigoplus_{i = 0}^n \left(\Sigma^{-ik,i} C\tau^k\right)^{\oplus \binom{n}{i}}
\]
on which the codegeneracy maps all act as the identity on the unshifted copy of $C\tau^k.$ In particular, we obtain a (non-canonical) equivalence
\[
\Delta_n = S^{0,0} \oplus \bigoplus_{i = 1}^n \left(\Sigma^{-(i-1)k,i} C\tau^k\right)^{\oplus \binom{n}{i}}
\]
on which the codegeneracy maps all act as the identity on $S^{0,0}.$ Since all the other summands are acyclic, the codegeneracy maps must be weak equivalences.

Similarly, we obtain an equivalence between latching object $L_n\Delta$ and the direct sum
\[
L_n\Delta =  S^{0,0} \oplus \bigoplus_{i = 1}^{n-1} \left(\Sigma^{-(i-1)k,i} C\tau^k\right)^{\oplus \binom{n}{i}}
\]
so that we have a pushout square
\[
\xymatrix{
0 \ar[r] \ar[d] & L_n \Delta\ar[d] \\
 \Sigma^{-(n-1)k,n} C\tau^k \ar[r] & \Delta_n.
}
\]
implying that $L_n\Delta \to \Delta$ is a cofibration.
\end{proof}
We immediately obtain several corollaries:
\begin{cor}
Let $X$ be any cofibrant object in the $k$-projective model structure. A cylinder object for $X$ is given by $X \otimes \Delta.$
\end{cor}
\begin{proof}
The proof of the previous lemma is preserved under $X \otimes - .$ The only thing to check is that $X \otimes C\tau^k$ is weakly contractible, which follows because the map
\[
\tau^k: X_n/X_{n+k} \to X_{n-k}/X_n
\]
factors through $X_n/X_n$ and is therefore zero.
\end{proof}
\begin{cor}\label{kcofiber}
Let $f: X \to Y$ be a map of cofibrant filtered $\cat C$-objects. The cofiber $C_k(f)$ of $f$ in $D^k(\cat C)$ is given by the pushout
\[
\xymatrix{
X \ar[r]^f \ar[d] & Y \ar[d] \\
X \otimes C\tau^k \ar[r] & C_k(f) 
}
\]
\end{cor}
\begin{proof}
This follows from the fact that $X \otimes C\tau^k$ is weak equivalent to $0$, all the objects are cofibrant, and the map $X \to X \otimes C\tau^k$ is a cofibration. 
\end{proof}
\begin{cor}\label{ctweird}
The cofiber $C_k(\tau)$ of $\tau: S^{-1,0} \to S^{0,0}$ satisfies
\[
C_k(\tau) = C\tau^{k+1}
\]
\end{cor}
\begin{proof}
This follows from the pushout diagram
\[
\xymatrix{
S^{-1,0} \ar[r] \ar[d] & S^{0,0} \ar[d] \\
\Sigma^{-1,0} C\tau^k \ar[r] & C\tau^{k+1}.
}
\]
\end{proof}
\begin{cor}
We have a natural isomorphism
\[
[S^{n,t}, X]_{\cat D^k(\cat C)} = \pi_{n,t}^k(X).
\]
\end{cor}

\begin{proof}
Apply Corollary $6.1.11$ of \cite{mg} to the resolution $\Delta \to S^{0,0}.$
\end{proof}

\subsection{The $k$-derived category is compactly generated by spheres}
\begin{thm}\label{compactspheres}
The $\infty$-category $\cat D^k(\cat C)$ is stable. Moreover, if $\cat C$ is compactly generated by the spheres $S^{t}$, then $\cat D^k(\cat C)$ is compactly generated by the spheres $S^{n,t}.$ 
\end{thm}
\begin{proof}
Note that $\cat D^k(\cat C)$ inherits cocompleteness from $\cat C$. By Corollary \ref{kcofiber}, the suspension functor in $\cat D^k(\cat C)$ is given by $\Sigma^{-k,1}$, which is an equivalence, so stability follows from \cite[Proposition~1.4.2.27]{ha}.

By \cite[Remark~1.4.4.3]{ha}, compact generation of stable $\infty$-categories is detected on the underlying triangulated categories. In particular, it suffices to check the following two facts:
\begin{enumerate}
\item If $[S^{n,t},X]_{\cat D^k(\cat C)} = 0$ for all $s$ and $q$, then $X$ is a zero object of $\cat D^k(\cat C).$  
\item For any set $\{X^i\}$ of objects of $\cat D^k(\cat C)$, the natural map
\[
\coprod_i [S^{n,t}, X^i]_{\cat D^k(\cat C)} \to [S^{n,t}, \coprod_i X^i]_{\cat D^k(\cat C)}
\]
is an isomorphism of abelian groups. 
\end{enumerate}
The first is immediate, since the given condition implies that the map $X \to 0$ is a weak equivalence. The second follows from the compact generation of $\cat C$, which implies that we have
\begin{align*}
\coprod_i [S^{n,t}, X^i]_{\cat \Fil(\cat C)}  &= \coprod_i [S^{t}, (X^i)_n]_{\cat C}  \\
&=  [S^{t}, \coprod_i (X^i)_n]_{\cat C}  \\
&= [S^{n,t}, \coprod_i X^i]_{\cat \Fil(\cat C)}.  \\
\end{align*}
This is an equivalence of $\mathbb{Z}[\tau]$ modules, by which we get the corresponding statement for $\cat D^k(\cat C).$
\end{proof}

\section{The spectral sequence of a filtered $\cat C$-object}\label{sseqsection}

Filtered objects in stable $\infty$-categories are one of the most common sources of spectral sequences in homotopy theory. Given a filtered spectrum $X$, one constructs a spectral sequence starting with the homotopy groups of the ``associated graded'' spectrum $X \otimes C\tau.$ Under reasonable conditions, this spectral sequence converges to the colimit of the tower. Special cases of this spectral sequence include the Serre Spectral Sequence, Adams-Novikov Spectral Sequence, Grothendieck Spectral Sequence, and Slice Spectral Sequence.

In this chapter, we reconstruct the standard filtration spectral sequence associated to an object of $\Fil(\cat C).$ While this construction is well-known, the purpose of this chapter is to relate the spectral sequence to the various derived categories $\cat D^k(\cat C)$ we've constructed earlier in this paper. This is an essential ingredient in the two chapters, which will allow us to relate the special fiber of the deformation we construct to the $E^k$ page of the filtration spectral sequence. 
\subsection{Construction of the spectral sequence} 
The following is well-known:
\begin{thm}\label{sseqexists}
Let $X$ be a filtered object of a stable $\infty$-category $\cat C$. There is a spectral sequence $E^{n,t}_r(X)$ with
\[
E_{n,t}^1(X) = \pi_{n,t}^{0}(X \otimes C\tau)
\]
which converges conditionally to $\pi_{t}^{top}(X)$ as long as $\lim X = 0.$ The $d_k$ differential takes elements in degree $(n,t)$ to elements in degree $(n+k,t-1).$
\end{thm}
In this section we relate this spectral sequence to the various derived $\infty$-categories we've defined. Along the way, we will reprove this theorem as well as the following stronger (but still well-known) version:
\begin{thm}\label{sseqmult}
Let $X$ be a ring in $\Fil(\cat C)$. Then, the above spectral sequence has a multiplicative structure converging conditionally to the multiplicative structure on $\pi^{top}(X).$
\end{thm}

The existence part of the proof of Theorem \ref{sseqexists} is contained in the following two propositions. The first constructs the spectral sequence, and the second relates it to the $\infty$-categories we constructed in the previous section.
\begin{prop}\label{ceexists}
There is an extended (cohomological) Cartan-Eilenberg system with:
\begin{itemize}
    \item $H^*(i,j) = \pi_{-*+1}(X_i/X_j).$ 
    \item For $i \leq i'$, $j \leq j'$, we have $\eta: H^*(i,j) \to H^*(i',j')$ induced by the diagram
    \[
    \xymatrix{ X_{i'} \ar[r] \ar[d] & X_i \ar[d] \\
     X_{j'} \ar[r] & X_{j} \\
    }
    \]
    \item For $i \leq j  \leq k$ we have $\delta: H^*(i,j) \to H^{*+1}(j,k)$ induced by the fiber sequence
    \[
    X_j/X_k \to X_i/X_k \to X_i/X_j.
    \]
\end{itemize}
Here we set $X_{\infty} = \lim X$, $X_{-\infty} = \colim X$, and assume $*$ is $A$-graded.
\end{prop}
\begin{proof}
By Lemma 1.2.2.4 in \cite{ha}, we can construct what Lurie calls a ``gap" diagram $\bar X: \cat J \to \Fil(\cat C)$, where $\cat J$ is the partially ordered set of pairs $(m,n)$ with $m,n \in \mathbb{Z} \cup \{-\infty, \infty\}$ with $m \leq n$. The partial ordering is given by $(m,n) \leq (m',n')$ whenever $m \leq m'$ and $n \leq n'$, and the diagram satisfies:
\begin{enumerate}
\item $\bar X(m,n) = X_m/X_n.$
\item For $m \leq m'$ and $n \leq n'$ the square
\[
\xymatrix{
\bar X(m',n') \ar[r] \ar[d] & \bar X(m',n) \ar[d] \\
\bar X(m,n') \ar[r] & \bar X(m,n)
}
\]
is a pushout square.
\end{enumerate}
The functoriality and naturality axioms for a Cartan-Eilenberg system can be read directly off this diagram, while the long exact sequence comes from the fiber sequence 
\[
    X_j/X_k \to X_i/X_k \to X_i/X_j.
\]
\end{proof}
\begin{prop}\label{sseqform}
Suppose $X_\infty = 0$. Then, the $(r+1)$st derived exact couple of this Cartan-Eilenberg system takes the form
\[
\xymatrix{
\bigoplus_{t,s} \pi^r_{nt} X \ar[rr]^{\tau}  && \bigoplus_{t,s} \pi^r_{nt} X \ar[dl] \\
 & \bigoplus_n \pi^r_{nt}(X \otimes C\tau^{r+1}) \ar[ul] & 
}
\]
The exact couple is induced by the cofiber sequence (in $\cat D^k$) of Corollary \ref{ctweird}
\[
\Sigma^{-1,0}X  \xrightarrow{\tau} X \to X \otimes C\tau^{r+1}.
\]
\end{prop}
\begin{proof}
When $r = 1$, we have the exact couple
\[
\xymatrix{
\bigoplus_{t,s} H^t(s, \infty) \ar[rr]  && \bigoplus_{t,s} H^t(s, \infty)\ar[dl] \\
 & \bigoplus_{t,s} H^t(s,s+1)\ar[ul] & 
}
\]
Equivalently, this is the exact couple
\[
\xymatrix{
\bigoplus_{t,n} \pi_t X_n \ar[rr]^\tau  && \bigoplus_{t,n} \pi_t X_n \ar[dl] \\
 & \bigoplus_{t,n} \pi_t(X_n/X_{n+1}) \ar[ul] & 
}
\]
The $E^{r+1}$ term coming from a Cartan-Eilenberg system is always given by
\[
E^{r+1}= \bigoplus_n \im(H^*(n,n+r+1) \to H^*(n-r,n+1)
\]
which is equivalent to
\[
E^{r+1} = \bigoplus_n \im(\pi_t(X_n/X_{n+r+1}) \xrightarrow{\tau^r} \pi_t(X_{n-r}/X_{n+1}))
\]
So
\[
E^{r+1} = \bigoplus_n \tau^{r} \pi_t(X_n/X_{n+r+1}) = \bigoplus \pi_{n,t}^{r}(X \otimes C\tau^{r+1}).
\]

\end{proof}
In particular, we note that the $k$-exact maps ``see'' the $E^{k+1}$ page of the spectral sequence. 
\begin{cor}\label{kexactek}
A $k$-exact map $f:X \to Y$ induces an isomorphism $f: E^{k+1}_{**}(X) \to E^{k+1}_{**}(Y)$
\end{cor}
\begin{proof}
Apply the five lemma to the unrolled exact couple
\[
\xymatrix{
\pi^k_{**}(X) \ar[r]^\tau \ar[d]&  \pi^k_{**}(X)  \ar[r] \ar[d] & \pi^k_{**}(X \otimes C\tau^{k+1}) \ar[r] \ar[d]  &\pi^k_{**}(X) \ar[r]^\tau \ar[d]&  \pi^k_{**}(X) \ar[d] \\
\pi^k_{**}(Y) \ar[r]^\tau &  \pi^k_{**}(Y)  \ar[r]  & \pi^k_{**}(Y \otimes C\tau^{k+1}) \ar[r]   &\pi^k_{**}(Y) \ar[r]^\tau &  \pi^k_{**}(Y) 
}
\]
\end{proof}
The only part of Theorem \ref{sseqexists} remaining is (conditional) convergence. But this is automatic, because we have a short exact sequence 
\[
0 \to \text{lim}^1(\pi_*(X_n) \to \pi_*(\lim X_n) \to \lim \pi_*(X_n) \to 0,
\]
so if the middle term is zero the other two must be as well.

Most of the proof of Theorem \ref{sseqmult} is contained in the following lemma.\begin{lem}\label{pairing}
There is a natural pairing of spectral sequences
\[
\mu: E^{n,t}_r(X) \otimes E^{n',t'}_r(Y) \to E^{n+ n', t+t'}_r(X \otimes Y)
\]
inducing a pairing on the $E_\infty$ pages. 
\end{lem}
\begin{proof}
By work of Douady in \cite{douady11suite} (with English expositions in e.g. \cite{hedenlund2021multiplicative, helle2017pairings})
it suffices to provide a pairing on the corresponding Cartan-Eilenberg system satisfying axioms we will refer to as (SPP I) and (SPP II). Rather than stating these in general, we will state and prove the relevant claims in the context of this spectral sequence. 

To construct a pairing, we require maps
\[
\mu_r: \pi_{n,t}^{0}\left(X \otimes C\tau^{r}\right) \otimes \pi_{u,\ell}^{0}\left(Y \otimes C\tau^r\right) \to \pi_{n+u, t+\ell}^{0} \left(X \otimes Y \otimes C\tau^r\right) 
\]
which we define on elements $f \otimes g$ via the composition
\[
\mu_r(f \otimes g): S^{n +u, q + \ell} \xrightarrow{f \otimes g} X \otimes C\tau^r \otimes Y \otimes C\tau^r\xrightarrow{T} X \otimes Y \otimes C\tau^r \otimes C\tau^r \to X \otimes Y \otimes C\tau^r.
\]
Since $\mu_r(f \otimes g)$ is bilinear in $f$ and $g$, this is well-defined. We now have to check Douady's axioms.

\textbf{(SPP I):} For each $s' \leq s$ and $u' \leq u$, we must have a commutative square
\[
\xymatrix{
\pi_{n,t}^{0}\left(X \otimes C\tau^{r}\right) \otimes \pi_{u,\ell}^{0}\left(Y \otimes C\tau^r\right) \ar^{\mu_r}[r] \ar^{\eta \otimes \eta}[d] & \ar^\eta[d] \pi_{n+u, t+\ell}^{0} \left(X \otimes Y \otimes C\tau^r\right)  \\
\pi_{n',t}^{0}\left(X \otimes C\tau^{r}\right) \otimes \pi_{u',\ell}^{0}\left(Y \otimes C\tau^r\right) \ar^{\mu_r}[r]  &  \pi_{n'+u', t+\ell}^{0} \left(X \otimes Y \otimes C\tau^r\right).
}
\]
This follows from applying the general fact that $\mu_r$ is functorial in $X$ and $Y$ to the maps
\[
X \xrightarrow{\tau^{n - s'}} \Sigma^{n - s',0} X  \text{  and  } Y \xrightarrow{\tau^{u-u'}} \Sigma^{u - u',0} Y.
\]

\textbf{(SPP II):} In the following diagram, the diagonal composition must be the sum of the two outside composites:
\begin{figure}[H]
\includegraphics[width=\textwidth-1mm]{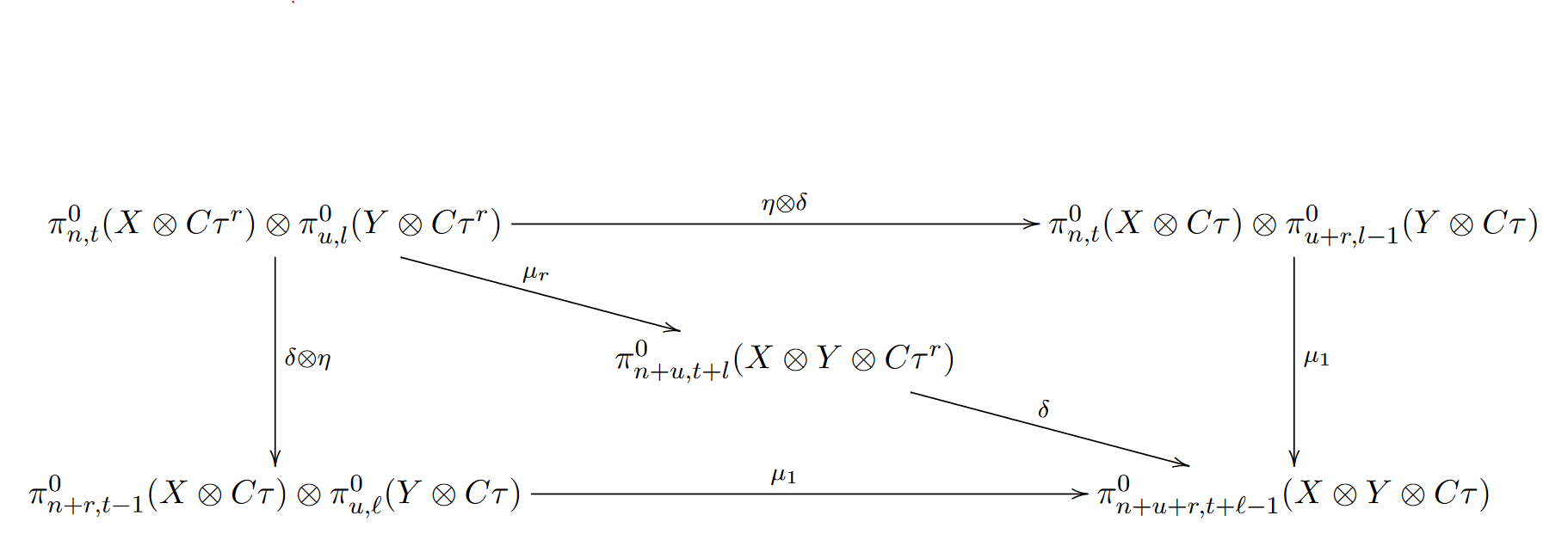}
\end{figure}

This follows from the corresponding diagram $\Fil(\cat C)$ satisfying the same property
\[
\xymatrix{
C\tau^r \otimes C\tau^r \ar^{\eta \otimes \delta}[rr] \ar^{\delta\otimes\eta}[dd]\ar^{\mu_r}[rd] && C\tau \otimes \Sigma^{-r,1} C\tau \ar^{\mu_1}[dd] \\ 
& C\tau^r \ar[rd]^{\delta} & \\
\Sigma^{-r,1} C\tau \otimes C\tau \ar[rr]^{\mu_1} &&\Sigma^{-r,1} C\tau,
}
\]
which is constructed as in the proof of Lemma \ref{deltader}.  
\end{proof}

\subsection{Example: the Adams spectral sequence}\label{adamssection}
In this section we express the Adams Spectral Sequence as the spectral sequence associated to a certain multiplicative filtration, which we will use to relate our results on matric Massey products to a spectral sequence of interest. Write $\FF_p$ for the Eilenberg-Maclane spectrum $\text{H}\FF_p$. Fix a spectrum $X$ and an $\FF_p$ Adams resolution $X_\bullet$ given by
\[
\xymatrix{
X = X_0 \ar[d] & X_1 \ar[l]\ar[d] & X_2\ar[l]\ar[d] & \cdots \ar[l] \\
K_0 & K_1 & K_2 & \cdots
}
\]
Let $\gamma X$ be the filtered spectrum
\[
\cdots \to X \to X \to 0 \to 0 \to \cdots
\]
with copies of $X$ in every positive degree connected by identity morphisms, and set
\[ Q X = \Sigma^{-1,0}(\gamma X /X_\bullet)
\]
so in coordinates we have
\[
Q X = \cdots \to  X_0/X_3 \to  X_0/X_2 \to X_0/X_1 \to \cdots.
\]
with $X_0/X_1$ in degree $0$. 
The canonical map 
\[
\gamma X \to QX
\]
dualizes to a map
\[
F(QX , \FF_p) \xrightarrow{\alpha} \iota F(X , \FF_p),
\]
\begin{prop}
The map $\alpha$ exhibits $F(QX , \FF_p)$ as a cofibrant replacement for $\iota F(X, \FF_p)$ in the $1$-projective model structure on $\Ch(\End(\FF_p)\Mod)$ with ``spheres'' $S^n = \Sigma^n \End(\FF_p).$
\end{prop}
\begin{proof}
We first check that $\alpha$ is an acyclic fibration. Consider the commutative square
\[
\xymatrix{
F(\Omega X_0/X_n, \FF_p) \ar[r] \ar[d] & F(X_0/X_{n+1},\FF_p) \ar[d] \\
F(X_n, \FF_p) \ar[r] & F(X_0, \FF_p).
}
\]
Applying $\pi_{-t}$ to this square gives
\[
\xymatrix{
H^t(X_0/X_n) \ar[r] \ar[d] & H^t(X_0/X_{n+1}) \ar[d] \\
H^t(X_0) \ar[r] & H^t(X_0).
}
\]

Since $X_\bullet$ is an Adams resolution, we know that $H^t(X_n) \to H^t(X_{n+1})$ is always zero. In particular, this implies that we have a natural splitting
\[
H^t(X_0/X_n) = H^q(X_0) \oplus H^{t-1}(X_n),
\]
so our square is just 
\[
\xymatrix{
H^t(X_0) \oplus H^{t-1}(X_s) \ar[r] \ar[d] & H^t(X_0)  \oplus H^{t-1}(X_{n+1}) \ar[d] \\
H^n(X_0) \ar[r] & H^n(X_0).
}
\]
In particular, this square maps the image of the top map isomorphically onto the image of the bottom map, so $\alpha$ is a weak equivalence. Moreover, the vertical maps are surjective, so $\alpha$ is a fibration.

Therefore  we only need to check that $F(QX, \FF_p)$ is cofibrant. Since its filtration is bounded above, it suffices to check that $(QX, \FF_p)_{s}/(QX, \FF_p)_{s+1}$ is always a direct sum of copies of $\End(\FF_p)$. But this follows because $X_s/X_{s+1}$ is always a direct sum of copies of $\FF_p.$
\end{proof}
We have proven:
\begin{cor}
Let $\CC^R(X)$ denote any cofibrant replacement in the $1$-projective model structure on filtered right $R$-modules, whose ``spheres'' are $S^n = \Sigma^n R$. Then, for any $X$, the spectral sequence associated to 
\[
F(\CC^{\End(\FF_p)}(F(X,\FF_p)),\FF_p)
\]
is the (additive) Adams Spectral Sequence for $X$.
\end{cor}
Since we are interested in multiplicative structure, we make a stronger claim in the case $X = S^0$: 
\begin{thm}
Keep the setting of the previous corollary. The spectral sequence associated to 
\[
F^{\Fil(\Mod_{\End(\FF_p)})}(\CC^{\End(\FF_p)}(\FF_p),\CC^{\End(\FF_p)}(\FF_p))
\]
is the Adams spectral sequence for the sphere, including its multiplicative structure.
\end{thm}
The proof of this is contained in the following two lemmas.
\begin{lem}
The map $\CC^{\End(\FF_p)}(\FF_p) \to \FF_p$ induces an map of spectral sequences
\[
E^k_{**}\left(F^{\Fil(\Mod_{\End(\FF_p)})}\left(\CC^{\End(\FF_p)}(\FF_p),\CC^{\End(\FF_p)}(\FF_p)\right)
\right) \to E^k_{**}\left(F\left(\CC^{\End(\FF_p)}(\FF_p),\FF_p\right)\right)
\]
which is an isomorphism starting at the $E_2$ page.
\end{lem}
\begin{proof}
This follows immediately because the cofibrant replacement map is $1$-exact on the underlying filtered spectra, so the map 
\[
F^{\Fil(\Mod_{\End(\FF_p)})}\left(\CC^{\End(\FF_p)}(\FF_p),\CC^{\End(\FF_p)}(\FF_p)\right)\to F\left(\CC^{\End(\FF_p)}(F(X,\FF_p),\FF_p\right)
\]
is also $1$-exact, and therefore an isomorphism on spectral sequences starting at the $E_2$-page by Corollary \ref{kexactek}.
\end{proof}
\begin{lem}
The multiplicative structure on 
\[F^{\Fil(\Mod_{\End(\FF_p)})}\left(\CC^{\End(\FF_p)}(\FF_p),\CC^{\End(\FF_p)}(\FF_p)\right)\] induces the usual multiplication on the $E_2$ page $\Ext_{\cat A}(\FF_p,\FF_p).$
\end{lem}
\begin{proof}
By unwinding definitions, we see that this is the usual composition multiplication on  $\Hom^*_{\cat A}(Q\FF_p, Q\FF_p)$ for the cofibrant replacement $Q\FF_p$ coming from $\CC^{\End(\FF_p)}(\FF_p)$. That this agrees with the Yoneda product is classical: see e.g. the discussion after Lemma 1 of \cite{sjodin1976set}.
\end{proof}

\subsection{Example: Ext and Tor spectral sequences}\label{exttorsec}
In this brief section we record, with little commentary, two special cases of this spectral sequence that we will use throughout the rest of the paper. 
\begin{thm}\label{torfil}
Let $X$ and $Y$ be modules over some ring $R$ in $\cat C$, and suppose we have filtrations $X_\bullet, Y_\bullet$, and $R_\bullet$ on each of the three. Then, there are spectral sequences
\[
E^1_{nt} = \pi_{**}^0\left(\left(X \otimes C\tau\right) \otimes_{R \otimes C\tau} \left(Y \otimes C\tau\right)\right)
\]
\end{thm}
\begin{proof}
The given filtrations place a filtration on the bar construction $B(X,R,Y)$. To identify the $E^1$ page, note that we have
\begin{align*}
 C\tau \otimes |B(X,R,Y)| &= |C\tau \otimes B(X,R,Y)| \\
                          &= |B(C\tau, C\tau, C\tau) \otimes B(X,R,Y)| \\
                          &= |B(X \otimes C\tau, R \otimes C\tau, Y\otimes C\tau)| \\
                          &= \left(X \otimes C\tau\right) \otimes_{R \otimes C\tau} \left(Y \otimes C\tau\right)
\end{align*} 
\end{proof}

\begin{thm}[Elmendorf-Kriz-Mandell-May \cite{ekmm}]\label{torclassic} 
Let $X$ and $Y$ be modules over some ring $R$ in $\cat C$. There are  spectral sequences
\[
E^2_{r*} = \Tor^r_{\pi_*R}(\pi_*(X), \pi_*(Y)) \Rightarrow \pi_{r + *}(X \otimes_R Y).
\]
and
\[
E^2_{r*} = \Ext^r_{\pi_*R}(\pi_*(X), \pi_*(Y)) \Rightarrow \pi_{r+*}(F_R(X, Y).
\]
The first converges absolutely, while the second converges conditionally. 
\end{thm}
\begin{proof}
Let $Q$ be the cofibrant replacement functor in $\cat D^1(R\Mod).$ These are just the spectral sequences associated to $QX \otimes Y$ and $\Hom(QX, Y)$, respectively.
\end{proof}

\section{Multiplicative Structure}\label{mmpsection}
In this chapter we relate the symmetric monoidal structures on $\cat C,$ on $\cat D^k(\cat C),$ and on an algebraic $\infty$-category relating to the $E^{k+1}$ page of the filtration spectral sequence. 

More specifically, we show that under mild conditions $\cat D^k(\cat C)$ can be viewed as a deformation whose generic fiber is $\cat C$ and whose special fiber is the derived $\infty$-category $\cat D(\Gr_k(\cat C^{\heart}))$, where $\Gr_k(\cat C^{\heart})$ is the (abelian) category of graded objects of $\cat C^{\heart}$ with grading in $\cat A \otimes \{0, \cdots, k-1\}$. The functor $\cat D^k(\cat C) \xrightarrow{\tau^{-1}} \cat C$ is just the colimit map for any representing filtration, while the functor $\cat D^k(\cat C) \to \cat D\left(\cat \Gr_k(C^{\heart})\right)$ takes $X$ to $E^{k+1}_{**}(X).$ 

The idea of ``a deformation of a stable $\infty$-category" has been formalized in several inequivalent ways. Throughout this paper, we use the term loosely to refer to $\infty$-categories satisfying analogues of conditions 1 through 4, above, which appears to match how Gheorghe et al. and Pstragowski use the term. However, we caution that Barkan \cite{barkan2023chromatic} has defined a deformation to be a $\Fil(\Sp)$-module in the $\infty$-category of presentable $\infty$-categories with colimit-preserving maps, while Burklund, Hahn, and Senger \cite{burklund2020galois} place additional restriction on the deformation's generators. We warn the reader that we do not believe our constructions satisfy these more stringent definitions.\footnote{While our $\infty$-category $\cat D^k(\cat C)$ is a $\Fil(\Sp)$-module in the $\infty$-category of all $\infty$-categories, some of the maps involved do not preserve colimits, a difference that is necessary for our deformation to directly recover the $E_k$ page of a spectral sequence for $k > 2.$}

In the rest of the chapter we argue that the symmetric monoidal structure on each of these functors provides a recipe for transferring higher multiplicative information in the filtration spectral sequence to higher multiplicative information in the spectrum the filtration converges to. As a case in point, we prove a generalization of Moss's convergence theorem, showing that the deformation we've constructed lets us apply an argument of Burklund to a much wider range of spectral sequences. 
\subsection{The $k$-derived category's deformation structure}\label{deformsec}
Recall from the introduction that the phrase ``deformation of stable $\infty$-categories'' is used inconsistently in the literature. Noting that we have already established that $\cat D^k(\cat C)$ is compactly generated by shifts of spheres, we observe in this section that we have a diagram of symmetric monoidal left adjoints
\[
\xymatrix{
 & \cat D^k(\cat C) \ar_{\tau = 0}[dl] \ar^{\tau^{-1}}[dr] & \\
\Mod_{C\tau^{k+1}}(\cat D^k(\cat C)) & & \cat C
}
\]
in which we can identify the special fiber $\Mod_{C\tau^{k+1}}(\cat D^k(\cat C))$ with an algebraic $\infty$-category and the map we've labelled ``$\tau = 0$" takes an object $X$ to the $E^{k+1}$ page of its corresponding spectral sequence. 

This diagram exists for purely formal reasons; the hard part of this is the identification of the fibers. To this end, we construct a $t$-structure on $\cat D^k(\cat C)$. and use it to understand $\Mod_{C\tau^{k+1}}(\cat D^k).$  In particular, we show that there is an equivalence between the subcategories of connective objects which then extends to the equivalence of interest.

\subsubsection{The filtration t-structure} 
In this section, we place a $t$-structure on $\cat D^k(\cat C)$ that will help us to analyze the deformation and identify its heart. 

We grade our t-structures \emph{cohomologically}, so that $\cat C^{\leq 0}$ denotes the connective part of $\cat C.$ 
\begin{defn}
For $k \geq 0$, the \emph{filtration t-structure} on $\cat D^k(\cat C)$ is defined by:
\begin{itemize}
    \item $D^k(\cat C)^{\leq 0}$ is the full subcategory of objects $X$ such that $\pi^k_{nt}(X)$ vanishes for $n > 0.$
    \item $\cat D^k(\cat C)^{\geq 0}$ consists of objects $X$ that  $\pi^k_{nt}(X)$ vanishes for $n < 0.$ 
\end{itemize}
\end{defn}

To see that this is a t-structure, we will need to apply ideas of Keller and Vossieck \cite{keller1988aisles}, expanded upon by Tarrio et al. \cite{tarrio2003construction}, who find technical conditions guaranteeing that a subcategory of a triangulated category is the connective part of a t-structure (for example, if the inclusion functor has a right adjoint.) We will use the following $\infty$-categorical reformulation of their ideas, due to Lurie:
\begin{lem}\label{luriet}
Let $\cat C$ be a presentable stable $\infty$-category. If $\cat C'$ is a full subcategory of $\cat C$ which is presentable, closed under small colimits, and closed under extensions, then there exists a t-structure on $\cat C$ such that $\cat C' = \cat C^{\leq 0}.$
\end{lem}
\begin{proof}
This is Proposition 1.4.4.11 of \cite{ha}.
\end{proof}

We proceed in two steps. First, in Lemma \ref{connectivet}, we check conditions guaranteeing that there is a symmetric monoidal t-structure on $\cat D^k(\cat C)$ with the desired connective piece. Second, in Theorem \ref{tstructure}, we check that $\cat D^k(\cat C)_{\leq 0}$ matches the description we gave and identify the t-structure's heart. 

\begin{lem}\label{connectivet}
Let $\cat P$ denote the sub-$\infty$-category $\cat D^k(\cat C)^{\leq 0}$. $\cat P$ has the following properties:
\begin{enumerate}
    \item $\cat P$ is closed under extensions and small colimits.
    \item $\cat P$ is presentable.
    \item $\cat P$ is closed under the symmetric monoidal product on $\cat D^k(\cat C).$ 
\end{enumerate}
\end{lem}
\begin{proof}
(1) Since $\cat D^k(\cat C)$ is compactly generated by spheres, the natural equivalence 
$$\pi^k_{nt}(X) \simeq [S^{n,t}, X]_{\cat D^k(\cat C)}$$
implies that for any collection $\{X_i\}_{i \in \cat I}$ of objects, the natural map 
\[
\coprod_{i \in \cat I} \pi^k_{**}(X_i) \to \pi^k_{**}\left( \coprod_{i \in \cat I} X_i \right)
\]
is an isomorphism, so $\cat P$ is closed under arbitrary coproducts. Now, suppose $X \to Y \to Z$ is a cofiber sequence in $\cat D^k(\cat C)$. We obtain a long exact sequence
\[
\cdots \to \pi^k_{nt} C\tau^{k+1} \otimes X \to  \pi^k_{nt} C\tau^{k+1} \otimes Y \to  \pi^k_{nt} C\tau^{k+1} \otimes Z \to  \pi^k_{n+k, t-1} C\tau^{k+1} \otimes X \to \cdots,
\]
by which we immediately see that if $X$ and $Z$ are in $\cat P$, then $Y$ must be as well. Similarly, note that if $n > 0$ we must also have $n + k > 0,$ so in particular if $X$ and $Y$ are in $\cat P$ then $Z$ must be as well. Since $\cat P$ is closed under cofibers and arbitrary coproducts, it must be closed under all colimits. 

(2) $\cat P$ is compactly generated by the set of spheres $S^{n,t}$ with $n \leq 0$, and is therefore presentable. 

(3) follows from the (trigraded) spectral sequence of \ref{exttorsec} with
\[
E^2 = \Tor^r_{\pi^k_{nt}S^{0,0}}(\pi^k_{nt}(X), \pi^k_{nt}(Y))
\]
converging to
\[
\pi^k_{n-kr,t+r}( X \otimes Y).
\]

\end{proof}

With this lemma in hand we obtain a $t$ structure, and we can immediately identify its heart.  

\begin{thm}\label{tstructure}
The filtration  t-structure is a symmetric monoidal t-structure, and its heart is the abelian category $\Gr_k\left(\cat C^{\heart}\right)$ of graded objects of $\cat C^\heart,$ where the grading takes values in $\cat A \times \{0, \cdots , k-1\}.$
\end{thm}
\begin{proof}
By Lemmas \ref{luriet} and \ref{connectivet}, there is a t-structure on $\cat D^k(\cat C)$ such that $\cat D^k(\cat C)^{\leq 0} = \cat P$. Further, Lemma \ref{connectivet} shows that this t-structure is compatible with the symmetric monoidal structure on $\cat D^k(\cat C)$, so it remains to identify the coconnective objects and the heart.

First, note that $\cat D^k(\cat C)^{\geq 1}$ is precisely the set of all $Y$ such that $[X,Y]_{\cat D^k(\cat C)} = 0$ for all $X$ in $\cat D^k(\cat C)^{\leq 0}.$ In particular, setting $X = S^{n,t}$, we must have
\[
\pi^k_{n,t}(Y) = 0
\]
whenever $s \leq 0.$ 

Conversely, suppose $\pi^k_{n,t}(Y) = 0$ for all $n \leq 0$, and suppose $X$ is in $\cat D^k_{\geq 0}(\cat C).$ We have a spectral sequence with \[E^2_{***} =
\Ext_{\pi^k_{**}S^{0,0}}\left(\pi^k_{**}(X),\pi^k_{**}(Y)\right)
\]
converging conditionally to
\[
[X,Y]_{\cat D^k(\cat C)}^{*,*}.
\]
Since $\pi^k_{**}X$ is concentrated in nonnegative $n$-degree and $\pi^k_{**}Y$ is concentrated in negative $n$-degree, it follows that \[
\Ext^r_{\pi^k_{**}S^{0,0}}\left(\pi^k_{**}(X),\pi^k_{**}(Y)\right)
\]
is concentrated in degrees $n < 0$ and $r \geq 0.$ The grading in the convergence takes the form
\[
(r,n,t) \mapsto (-kr + n, t + r),
\]
so
\[
[X,Y]_{\cat D^k(\cat C)}^{*,*}.
\]
is concentrated in negative $n$-degree, and in particular
\[
[X,Y]_{\cat D^k(\cat C)}^{0,0} = 0,
\]
as desired. 

We next identify the heart. There is an obvious map
\[
\cat D^k(\cat C)^\heart \to \Gr_k\left(\cat C^{\heart}\right)
\]
taking $X$ to $\pi^k_{nt}(X)$ for $0 \leq n < k$. Using shifts of objects in $\cat C^\heart$ it is easy to see that this map is essentially surjective. To see that it is fully faithful, we use the 
spectral sequence from Section \ref{exttorsec}
\[
\Ext^r_{\pi^k_{**}S^{0,0}}\left(\pi^k_{**}(X),\pi^k_{**}(Y)\right)
\]
converging conditionally to
\[
[X,Y]_{\cat D^k(\cat C)}^{*,*}.
\]
The grading takes the form
\[
(r,n,t) \mapsto (-kr + n, t+ r)
\]
If $X$ and $Y$ are both in $\cat D^k(\cat C)^\heart,$ the spectral sequence collapses for degree reasons and we have 
\[
[X,Y] = \Hom_{\pi^k_{**}S^{0,0}}\left(\pi^k_{**}(X),\pi^k_{**}(Y)\right),
\]
as desired.
\end{proof}
As an immediate consequence we have:
\begin{cor}
The homotopy group $\pi_0(X)$ coming from the filtration t-structure can be identified with the direct sum of the $\pi^k_{nt}(X)$ for $0 \leq n \leq k.$
\end{cor}
\subsubsection{Identifying the fibers of the deformation}
In this subsection, we identify the special and generic fibers of this deformation.  We will see that the hard work came in building the t-structure: the structure of the special fiber will be a direct consequence. First, we quickly identify the generic fiber.
\begin{thm}
There is a natural symmetric monoidal equivalence between $\cat D^k(\cat C)[\tau^{-1}]$ to $\cat C.$
\end{thm}
\begin{proof}
Since localizations commute with one another, we can write
\[
\cat D^k(\cat C)[\tau^{-1}] = \Fil(\cat C)[\tau^{-1}][\cat W_k^{-1}].
 \]
But $\Fil(\cat C)[\tau^{-1}]$ just recovers $\cat C$, and maps in $W_k$ become equivalences in this localization. Monoidality follows by the same argument after restricting $\Fil(\cat C)$ to cofibrant objects. 
\end{proof}
Now we identify the special fiber. To this end, we recall the following definitions from \cite{pstrkagowski2022abstract}.
\begin{defn}
A commutative algebra $A$ in $\cat D^k(\cat C)^{\leq 0}$ is called \emph{shift} if the map  $\tau^k: \Sigma^{-k,0}X \to X$ inducing an isomorphism
\[
\pi^k_{**}(A) = \pi^k_{[0,k-1],*}(A)[\tau^k]
\]
An $A$-module $M$ is called \emph{periodic} if $\pi^k_{**}(M) =  \pi^k_{[0,k-1],*}(M)[\tau^k].$
\end{defn}
Throughout this section we consider the shift algebra $A = S^{0,0}$, noting that any object in $\cat D^k(\cat C)^{\leq 0}$ comes equipped with an essentially unique $A$-module structure. We require two very brief lemmas:

\begin{lem}
The $\infty$-category $\cat D^k(\cat C)^{\leq 0}$ is compactly generated by $S^{n,t}$ with $n < 0.$  
\end{lem}
\begin{proof}
This is immediate because $S^{n,t}$ generate $\cat D^k(\cat C)^{\leq 0}$ and the filtration t-structure is presentable. 
\end{proof}
\begin{cor}\label{colimgenper}
The $\infty$-category $\cat D^k(\cat C)^{\leq 0}$ is generated under colimits by periodic objects
\end{cor}
\begin{proof}
In the previous lemma  it is enough to take $S^{n,t}$ for $-k < n < 0$, since the others are generated by successive suspensions. These spheres are periodic.
\end{proof}
Now at last we can identify the special fiber.
\begin{thm}\label{symmonequiv}
Suppose $\cat C^\heart$ has enough projectives. There is a symmetric monoidal equivalence 
\[
\Mod_{C\tau^{k+1}}\left(\cat D^k(\cat C)^{\leq 0}\right) \to \cat D(\Gr \cat C^\heart)^{\leq 0}
\]
\end{thm}
that takes $X \otimes C\tau^{k+1}$ to the $E^k$ page of the spectral sequence associated to $X$. 
\begin{proof}
The underlying equivalence follows from Theorem 3.11 of \cite{pstrkagowski2022abstract}, noting that $(S^{0,0})^{\geq 0} = C\tau^{k+1}$ and applying Corollary \ref{colimgenper}.

To show that this can be promoted to a(n essentially unique) symmetric monoidal equivalence, it suffices to show that there is an essentially unique symmetric monoidal structure on $\cat D(\Gr \cat C^\heart)^{\leq 0}$  such that (1) the canonical symmetric monoidal structure on $\Gr \cat C^\heart$ and (2) . By Lemma 2.60 of \cite{pstrkagowski2023synthetic}, $\cat D(\Gr \cat C^\heart)^{\leq 0}$ is equivalent to the $\infty$-categorical of spherical presheaves (of spaces) on the $\infty$-category of projectives in $\Gr \cat C^\heart.$ The result then follows from Corollary 2.29 in the same paper.
\end{proof}
\begin{cor}
Suppose $\cat C^\heart$ has enough projectives. There is a symmetric monoidal equivalence 
\[
\Mod_{C\tau^{k+1}}\left(\cat D^k(\cat C)\right) \to \cat D(\Gr_k \cat C^\heart)
\]
\end{cor}
\begin{proof}
Take the $\infty$-category of spectrum objects on each side of the equivalence in Theorem \ref{symmonequiv}.
\end{proof}

\subsection{An $\infty$-categorical characterization of matric Massey products}\label{realmmpsection}
\subsubsection{Review of chain complexes in stable $\infty$-categories}\label{complexessection}
Filtered objects are sometimes casually referred to as ``the stable $\infty$-categorical version of chain complexes"\footnote{Source: the author has referred to them this way in the past.}. This correspondence was made explicit by Ariotta \cite{ariotta2021coherent}. In this brief section we define chain complexes in stable $\infty$-categories and state Ariotta's theorem, since this language will be more natural when we discuss matric Massey products in the next section. 
\begin{defn}
Let $\Ch$ denote the category with objects $[i]$ and $*_i$ for nonnegative integers $i$ and with a unique arrow $*_i \to [j]$, $[i] \to *_j$, $[i] \to [j]$, and $*_i \to *_j$ for $i >  j.$ 
\[
\xymatrix{
\cdots \ar[r] & \ast_4 \ar[r] \ar[rd]  & \ast_3 \ar[r] \ar[rd] & \ast_2 \ar[r] \ar[rd] & \ast_1 \ar[r] \ar[rd] & \ast_0  \\
\cdots \ar[r] &  [4] \ar[r] \ar[ru] & [3] \ar[r] \ar[ru] & [2] \ar[r] \ar[ru] & [1] \ar[r] \ar[ru] &  [0]
}
\]

A (connective) \emph{chain complex in $\cat C$} is a functor $\Ch \to \cat C$ taking $*_i$ to the zero object in $\cat C$ for all $i$. The $\infty$-category $\Ch(\cat C)$ is the full sub-$\infty$ category of $\Fun(\Ch, \cat C)$ whose objects are chain complexes.
\end{defn}

If we replaced $\cat C$ with an (ordinary) abelian category, this would recover the usual definition of chain complexes. In the $\infty$-category setting, the definition builds in coherent homotopies witnessing the vanishing of certain Toda brackets.

There is a natural $\infty$-categorical analogue of a chain complex's cycle groups $Z_n(X)$:
\begin{defn}
For $0 \leq i < j \leq \infty$, let $\Ch_i^j$ denote the full subcategory of $\Ch$ with objects $[k]$ and $*_k$ for $i \leq k \leq j.$   Given a functor $X: \Ch_i^j \to \cat C$, we let $Z_i^j(X)$ denote the limit of $X$. Similarly, let $C_i^j(X)$ denote the corresponding colimit. We will mildly abuse notation and write $Z_{i}^j(X)$ to denote $Z_i^{j}(X|_{\Ch_i^j})$ if $X$ is defined on a larger category (typically $\Ch$).  
\end{defn}
\begin{prop}\label{fib}
There are fiber sequences $Z_i^j(X) \to X_j \to Z_i^{j-1}(X)$ and cofiber sequences $C_{i+1}^j(X) \to X_i \to C_i^j(X)$
\end{prop}
\begin{proof}
The first is a special case of Proposition 4.4.2.2 of \cite{htt}, with $L$ the full subcategory excluding $[j]$ and $K'$ the full subcategory excluding $*_j.$ The second is dual.
\end{proof}
In the case $j = i + 1$, we see that $Z_i^j(X)$ is the fiber of the map $X_j \to X_i$, while $C_i^j(X)$ is its cofiber, so that we have $C_i^j(X) = \Sigma Z_i^j(X).$ In Appendix \ref{appiter}, we prove the following proposition generalizing this fact:   
\begin{prop}\label{iter}
Iterated fibers agree with iterated cofibers after a shift. More explicitly, for any $X: \Ch_i^j \to \cat C$ we have a natural equivalence
\[
C_i^j(X) = \Sigma^{j-i}Z_i^j(X).
\]
\end{prop}
\begin{thm}[Ariotta]\label{ariottathm}
Suppose $\cat C$ is closed under sequential limits. Then, there is an equivalence between the $\infty$-category of chain complexes and the $\infty$-category of filtered objects $X$ with $X_\infty = 0.$ In our notation, this equivalence takes a chain complex \[
\cdots \to Y_2 \to Y_1 \to Y_0 \to Y_{-1} \to \cdots
\]
to the filtered object
\[
\cdots \to \Sigma^{-1}Z_{-\infty}^{-1}(Y) \to Z_{-\infty}^{0}(Y)\to \Sigma Z_{-\infty}^{1}(Y) \to \cdots
\]
\end{thm}
\begin{proof}
This is Theorem 4.7 of \cite{ariotta2021coherent}.
\end{proof}

\subsubsection{Chain complexes in the derived $\infty$-category}
In this section we apply the theory of the previous section to the specific case of the derived $\infty$-category of a(n ordinary) commutative ring $R$. We will see that the higher homotopies in the definition of a chain complex correspond to maps satisfying similar properties to May's notion of \emph{matric Massey products.} This will allow us to define matric Massey products using purely $\infty$-categorical data, which will help us in the next sections to pass higher multiplicative information along symmetric monoidal $\infty$-categorical functors. 

The idea that something like this ought to be true appears to be folklore, but we were unable to find a source that either explicitly states or proves the relevant facts. Throughout this section, let $R$ be a commutative ring. Throughout, we make the following sign conventions:
\begin{conv}
The shifted chain complex $X[i]$ has differential $(-1)^i d_X.$
\end{conv}
\begin{conv}
The mapping cone of $f: X \to Y$ has the same elements as $X[1] \oplus Y$ and differential $-d_X + d_Y + f.$
\end{conv}
Our main result provides an explicit construction of chain complexes in $\cat D(R).$
\begin{prop}\label{drchains}
Let $X_n \xrightarrow{f_n} \cdots \xrightarrow{f_2} X_1 \xrightarrow{f_1}  X_0$ be a series of maps of cofibrant chain complexes of $R$-modules. Any extension of this diagram to a chain complex in $\cat D(R)$ comes from a series of maps (not necessarily chain maps) $f_{ij}: X_j \to X_i[j-i-1]$ for $0 \leq i < j \leq n$ such that:
\begin{itemize}
\item $f_{i-1,i} = f_i$.
\item For any cycle $x$, $df_{i,k}(x) = \sum_{j = i+1}^{k-1} (-1)^{j+1}f_{i,j}f_{j,k}(x)$
\end{itemize}
Given this information, the object $Z_0^n(X)$ is represented by a chain complex
\[
\left(Z_0^n(X)\right)_k = (X_0)_{k+n} + (X_1)_{k+n-1} + \cdots + (X_n)_{k}
\]
with differential 
\[
d(x_0,x_1,\cdots, x_n) = (-1)^n\left(\sum_{j = 0}^n (-1)^{j}f_{0j}(x_j), \sum_{j=1}^n (-1)^{j}f_{1j}(x_j), \cdots ,(-1)^ndx_n \right)
\]
where we set $f_{jj}(x) = dx$ to simplify the notation.
\end{prop}
\begin{proof}
We proceed by induction. The case $n = 0$ is automatic. Now, an extension of a chain complex of length $n$ to one of length $n+1$ corresponds to a (homotopy class of) map(s) 
\[
X_{n+1} \to Z_0^n(X)
\]
such that the induced map $X_{n+1} \to X_n$ is $f_n$. Using the formula for $Z_0^n(X)$, a map $f: X_{n+1} \to Z_0^n(X)$ splits into maps $f_{i,n+1}: X_{n+1} \to X_i[n-i].$ For $f$ to be a chain map, we require $df = fd$. In particular, we must have
\begin{align*}
f(dx) &= \sum df(x)  \\
                   &=  (-1)^n\left(\sum_{j = 0}^n (-1)^{j}f_{0j}f_{j,n+1}(x), \sum_{j=1}^n (-1)^jf_{1j}f_{j,n+1}(x), \cdots ,(-1)^ndf_{n,n+1}(x) \right)
\end{align*}
Since $x$ is a cycle, we have $dx = 0$, so we must indeed have
\begin{align*}
df_{i,n+1}(x) &= \sum_{j = i+1}^{n} (-1)^{j+1}f_{i,j}f_{j,n+1}(x).
\end{align*}
Now, $Z_0^{n+1}(X)$ can be computed as the homotopy fiber of $f$. This has components
\begin{align*}
\left(Z_0^{n+1}(X)\right)_k &= (Z_0^n(X))_{k+1} \oplus (X_{n+1})_k\\ 
                            &= (X_0)_{k+n} + (X_1)_{k+n-1} + \cdots + (X_n)_{k}
\end{align*}
and differential
\begin{align*}
d(x_0,\cdots, x_{n+1}) &=  d^{X_{n+1}}(x_{n+1}) - d^{Z_0^n(X)}(x_0,\cdots, x_{n}) + f(x_{n+1}) \\
                    &=\Big(0,0,\cdots,0,dx_{n+1}\Big) \\
                    &\,\,\,\,\,\,\,- (-1)^{n}\left(\sum_{j = 0}^{n} (-1)^{j}f_{0j}(x_j),  \sum_{j=1}^{n} (-1)^{j}f_{1j}(x_j), \cdots ,  \sum_{j=n}^{n} (-1)^{j}f_{n,j}(x_j), 0\right) \\
                    &\,\,\,\,\,\,\,+\Big(f_{0,n+1}(x_{n+1}), f_{1,n+1}(x_{n+1}), \cdots, f_{n,n+1}(x_{n+1}),0\Big)\\
                    &= (-1)^{n+1}\left(\sum_{j = 0}^{n+1} (-1)^{j}f_{0j}(x_j),  \sum_{j=1}^{n+1} (-1)^{j}f_{1j}(x_j), \cdots ,  (-1)^{n+1}dx_{n+1}\right)
\end{align*}
\end{proof}
To relate this to May's notion of \emph{matric Massey products}, we require an additional sign convention. 
\begin{conv}[May, \cite{mmp}]
For an element $x$, we set $\bar x = (-1)^{1 + \text{deg}(x)}x$. For a matrix $X$, let $\bar X$ be the matrix formed by applying this operation to each of its elements.
\end{conv}
We start by recalling the definition of a matric Massey product from \cite{mmp}. Suppose we have an augmented differential graded algebra $U$, a left $U$-module $N$, and a right $U$-module $N$, and let $(V_1, \cdots, V_{\ell+1})$ be a series of multipliable matrices such that the elements of $V_i$ are cycles in $M$ if $i = 1$, cycles in $N$ if $i = \ell+1$, and cycles in $U$ otherwise. We inductively define matrices $A_{i,j}$ and for $0\leq i < j \leq \ell+1$, except for $(i,j) = (0,\ell+1)$, by setting
\[
A_{i-1,i} = V_i
\]
and
\[
dA_{i,j} = \tilde{A}_{i,j}
\]
where
\[ 
\tilde{A}_{i,j} = \sum_{k = i + 1}^{j-1} \bar A_{i,k}A_{k,j}.
\]
Notice that $A_{i,j}$ is not uniquely defined: there can be multiple choices (and often zero choices) for $A_{i,j}$ at each step.
\begin{defn}\label{mmp}[May, \cite{mmp}]
The matric Massey product $\langle V_0, \cdots, V_{\ell+1} \rangle$ is the set of all possible $\tilde{A}_{0, \ell+1}$ coming from different choices for $A_{i,j}$ in the above process. 
\end{defn}
We can relate this to the construction of Proposition \ref{drchains} by a two-step process.

Let $V_1, \cdots V_{\ell+1}$ be multipliable matrices as above. Consider the system
\[
\bigoplus_{d^1_i \in D^1} U[d^1_i] \xrightarrow{V_2} \bigoplus_{d^2_i \in D^2} U[d^2_i] \xrightarrow{V_3} \cdots \xrightarrow{V_{\ell}}\bigoplus_{d^{\ell}_i \in D^\ell} U[d^{\ell}_i],
\]
where the $D^k$s are collections of degrees (possibly with repetition) induced from the $V_i$s. Applying Proposition \ref{drchains}, and noting that a $U$-module map between shifted copies of $U$ is fully determined by the image of $1$, we see extensions of this sequence to a chain complex $X$ in $U$-modules are equivalent to choices $A_{i,j}$ as above for $1 \leq i < j \leq \ell$. 

Moreover, a lift of the map $$R \xrightarrow{V_1} M \otimes_U \bigoplus_{d_i^1 \in D^1} U[d_i^1]$$ to a map $$R \xrightarrow{V_1'} M \otimes_U Z_0^{\ell-1}X$$
is the same information as choices $A_{0,j}$ with $0 < j < \ell$, and a lift of the map
\[
\bigoplus_{d_i^\ell} U[d_i^\ell] \xrightarrow{V_{\ell+1}} N[d^{\ell + 1}]
\]
(where $d^{\ell+1}$ is the degree of $V_1\cdots V_{\ell+1}$) to a map
\[
C_0^{\ell-1}(X) \xrightarrow{V_{\ell+1}'} N[d^{\ell+1}]
\]
is the same information as choices $A_{i,\ell+1}$ with $i > 0.$  Finally, note that $\tilde{A_{0,\ell+1}}$ is the image of $1$ under the composite
\[
\xymatrix{
R[\ell - 1] \ar[r]^-{V_1'[\ell-1]}& M \otimes_U Z_0^{\ell-1}(X) [\ell-1] \ar[d]^{\cong}& \\
&Z_0^{\ell-1}(M \otimes_U X) [\ell-1] \ar[d]^\cong  & \\
&C_0^{\ell-1}(M \otimes_U X) \ar[r]^-{V_{\ell+1}'} & N[d^{\ell+1}].
}\]

We can repackage this as a definition.
\begin{defn}\label{smashtoda}
For any $\cat C$ satisfying the conditions of Section \ref{filtersection}, suppose we have a ring $U$ in $\cat C$, a right $U$-module $M$, and a left $U$-module $N$. Given a multipliable series of matrices $V_1,\cdots, V_{\ell+1}$ such that the elements of $V_1$ lie in in $\pi_*(M)$, the elements of $V_{\ell+1}$  lie in $\pi_*(N)$, and the elements of the other matrices lie in $\pi_*(U)$, we define a (possibly empty) set of elements we call the \emph{smash Toda bracket} and denote
\[
\langle V_1, \cdots, V_{\ell+1} \rangle \subseteq \pi_*(M \otimes_U N)
\]
as follows. First, consider each way of extending the series of maps
\[
\bigoplus_{d^1_i \in D^1} \Sigma^{d^1_i} U \xrightarrow{V_2} \bigoplus_{d^2_i \in D^2} \Sigma^{d^2_i} U \xrightarrow{V_3} \cdots \xrightarrow{V_{\ell}}\bigoplus_{d^{\ell}_i \in D^\ell} \Sigma^{d^\ell_i} U
\] to a chain complex $\mathbb{T}(V)$, and each lift of the maps $$S^{0} \xrightarrow{V_1} \bigoplus_{d^1_i \in D^1} M$$ and $$\bigoplus_{d^{\ell}_i \in D^\ell} \Sigma^{d^\ell_i} U \xrightarrow{V_{\ell + 1}} \Sigma^{d^{\ell +1}} N$$ to maps
$$S^0 \xrightarrow{V_1'} Z_0^{\ell-1}(M \otimes_U \mathbb{T}(V))$$
and $$C_0^{\ell-1}\left(\mathbb{T}(V)\right) \xrightarrow{V_{\ell+1}'} \Sigma^{d^{\ell +1}} N.$$ (If any of these are not possible, the bracket is by definition empty.) The smash Toda bracket then consists of compositions
\[
\xymatrix@=2.5cm{
S^{\ell-1} \ar[r]^-{\Sigma^{\ell-1} V_1'} & \Sigma^{\ell -1} Z_0^{\ell-1}(M \otimes_U \mathbb{T}(V)) \ar[d]^\cong & \\
& C_0^{\ell-1}(M \otimes_U \mathbb{T}(V)) \ar[r]^{M \otimes_U V_{\ell+1}'}  &\Sigma^{d^\ell+1} M \otimes_U N
}
\]

\end{defn}
The discussion leading up to this section has shown:
\begin{thm}
Let $R$ be an ordinary ring, and let $U$ be a differential graded $R$-algebra whose underlying differential graded module is cofibrant in the projective model structure. Then, matric Massey products and smash Toda brackets in $\cat D(R)$ coincide. 
\end{thm}

\subsection{Moss's convergence theorem}\label{mosssection}
As a simple demonstration of the power of this kind of deformation, we recall Burklund's proof of Moss's convergence theorem \cite[Section 4.2]{burklundsynthetic} for the Adams Spectral Sequence. Moss's convergence theorem relates two kinds of higher multiplicative structure: Massey products, which are defined in any differential graded algebra, and Toda brackets, which are defined in on the homotopy groups of ring spectra. More specifically, Moss showed that under certain conditions a Massey product in the $E_k$ page of the Adams Spectral Sequence can detect a corresponding Toda bracket in the homotopy groups of its target:
\begin{thmn}[Moss]
For spectra $X$, $Y$, $Z$, and $W$,  let $E_k^{n,t}(X,Y)$ denote the $E^k$ page of the mod p Adams spectral sequence computing $[X,Y^\wedge_p].$ Suppose we have permanent cycles $a \in E_k^{n,t}(X,Y)$, $b \in E_k^{n',t'}(Y,Z),$ and $c \in E_k^{n'',t''}(Z,W)$ such that $ab = 0$ and $bc= 0.$ Suppose $a,b,$ and $c$ detect maps of spectra $\omega, \omega', $ and $\omega''$ of degree $i, i',$ and $i''$ respectively such that $\omega \omega' = \omega' \omega'' = 0.$

Moreover, suppose the spectral sequences $E_k(X,Z)$ and $E_k(Y,W)$ satisfy the following \textbf{crossing differential hypotheses}:
\begin{itemize}
\item For $0 \leq n \leq s + s' - k$, every element of $E^{n, i + i' + n+1}_{n + s' - n +1}(X,Z)$ is a permanent cycle.
\item For $0 \leq n \leq s' + s'' - k$, every element of $E^{n, i' + i'' + n+1}_{n' + s'' - n +1}(Y,W)$ is a permanent cycle.
\end{itemize}
Then, the Massey product $\langle a,b,c \rangle$ contains a permanent cycle that detects an element of the Toda bracket $\langle \omega, \omega', \omega'' \rangle.$ 
\end{thmn}
Further work has generalized this observation to a wider class of higher multiplicative operations and spectral sequences. Lawrence's PhD Thesis \cite{lawrence1969matric} showed that Moss's theorem holds for all matric Massey products. May's seminal paper \cite{mmp} on matric Massey products proves a similar statement for the spectral sequence associated to a filtered differential graded algebra.  More recently, Belmont and Kong \cite{belmont2021toda} have proven the statement for triple Massey products in filtration spectral sequences in arbitrary symmetric monoidal stable topological model categories.

Burklund shows that, at least for the $E_2$ page of the Adams Spectral Sequence, Moss's theorem is an immediate consequence of the deformation structure on synthetic spectra. The proof consists of three quick observations:
\begin{enumerate}
\item Let $\nu$ denote the fully faithful functor from spectra to synthetic spectra. A class $\alpha$ in the Adams $E_2$ page for a spectrum $X$ is a permanent cycle detecting a homotopy class $a$ if and only there is an element $\tilde{\alpha} \in \pi_{**}(\nu X)$ specializing to $\alpha$ in $\pi_{**}(X \otimes C\tau)$ and $a$ in $\tau^{-1}X.$  
\item If $ab = 0$ in the Adams $E_2$ page, then $\tilde{a}\tilde{b}$ is $\tau^k$-torsion for some $k$.
\item The crossing differential hypothesis implies that any $\tau^k$ torsion element in the same degree as $\tilde{a}\tilde{b}$ mapping to zero in the Adams $E_2$ page must be equal to zero. 
\end{enumerate}
These three observations imply that, given $a,b,c$ in the Adams $E_2$ page with $ab = bc = 0$, we can construct lifts $\tilde{a}, \tilde{b}$, and $\tilde{c}$ with $\tilde{a}\tilde{b} = \tilde{b}\tilde{c} = 0$. Because the maps from synthetic spectra to its special and generic fiber are exact and symmetric monoidal, the Toda Bracket $\langle \tilde a, \tilde b, \tilde c \rangle$ specializes to both the Massey product $\langle a,b,c \rangle$ and the Toda bracket $\langle \omega, \omega', \omega'' \rangle$. A similar proof applies to more complex Massey products.

In this section, we construct a deformation of a sufficiently well-behaved symmetric monoidal stable $\infty$-category $\cat C$ whose special fiber corresponds with the $E_k$ page of the more general filtration spectral sequence. As an example application, we then show that Burklund's argument can be carried out in this setting, providing a more intuitive and general proof of Belmont and Kong's result.

To do this, we first prove a couple of lemmas relating homotopical information in $\cat D^k(\cat C)$ to corresponding statements about the spectral sequence.
\begin{lem}\label{survive}
Let $\alpha$ be an element of $E^k_{nt}(X)$. Then, $\alpha$ survives to the $E^j$ page if and only if $\alpha$ is in the image of the natural map
\[
\pi^k_{nt}(X \otimes C\tau^{j+1}) \to \pi^k_{nt}(X \otimes C\tau^{k+1}).
\]
Similarly, $\alpha$ survives to the $E^\infty$ page if and only if it is in the image of the natural map
\[
\pi^k_{nt}(X) \to \pi^k_{nt}(X \otimes C\tau^{k+1}).
\]
\end{lem}
\begin{proof}
We have a natural commutative square
\[
\xymatrix{
\pi_t(X_n/X_{n+j+1}) \ar[r] \ar[d]& \pi_t(X_{n-k}/X_{n+j-k+1}) \ar[d]  \\
\pi_t(X_n/X_{n+k+1}) \ar[r]   & \pi_t(X_{n-k}/X_{n+1}).
}
\]
The image of the bottom map is $E^k_{nt}(X)$, while the image of the top map is $\pi_{nt}^k(X \otimes C\tau^{j+1}).$ The claim follows because $E^j_{nt}(X)$ is the image of the composition of the top map with the map $$\pi_t(X_{n-k}/X_{n+j-k+1}) \to \pi_t(X_{n-j}/X_{n+1}).$$ The proof of the second part is similar.
\end{proof}
Notice that our statement of the crossing differential hypothesis is slightly different from Moss's, since we have graded things differently and do not assume our filtrations are bounded on either side.
\begin{defn}
Fix $n$ and $t$. We say that the \emph{crossing differential hypothesis} holds on the $E^k$ page in degrees $n,t$ if, for any $\ell > 0$, all elements of
\[
E^{k + \ell+1}_{n - k - \ell,t+1}(X)
\]
are permanent cycles.
\end{defn}
We follow Burklund by rephrasing Moss's Proposition 6.3  as a statement about $\tau$-power torsion:
\begin{lem}
Suppose the crossing differential hypothesis holds on the $E^k$ page in degree $n,t$. Then, the kernel of the map
\[
\pi^{k-1}_{nt}(X) \to \pi^{k-1}_{nt}(X \otimes C\tau^k) = E^k_{nt}(X)
\]
is $\tau^j$-torsion free for all $j > 0.$
\end{lem}
\begin{proof}
Suppose otherwise. Choose $x$ in $\pi^{k-1}_{nt}(X)$ mapping to zero in $E^k_{nt}(X)$, and choose the smallest $j$ such that $\tau^{j}x = 0.$ Using the exact couple in Proposition \ref{sseqform}, we must have $x = \tau y$ for some $y$ in $\pi^{k-1}_{n+1,t}(X).$ Notice that to prove $x = 0$ it suffices to show that $\tau^jy = 0,$ since this will imply that $\tau^{j-1}x = 0$ and contradict the assumption that $j$ was minimal. 

Since $\tau^{j+1}y = 0$, we can use the cofiber sequence (in $\cat D^{k-1}(\cat C)$)
\[
\Sigma^{-j-1,0}X \xrightarrow{\tau^{j+1}} X \to X \otimes C\tau^{j+k}
\]
to find a preimage $z$ of $y$ in $\pi^{k-1}_{n-j-k,t+1}(X \otimes C\tau^{j+k}).$ By the crossing differential hypothesis, $z$ is a permanent cycle and therefore the image of an element of $\pi^{k-1}_{n-j-k,t+1}(X),$ and so maps to zero in $\pi^{k-1}_{n-j+1,t}X.$

But the commutative diagram
\[
\xymatrix{
X \otimes C\tau^{j+k} \ar[r] \ar[d] & \Sigma^{-j-k+1,1}X \ar[d]^-{\tau^j} \\
X \otimes C\tau^{k} \ar[r] & \Sigma^{-k+1,1}X
}
\]
shows that the image of $z$ in $\pi^{k-1}_{n-j+1,t}X$ is equal to $\tau^jy,$ which must therefore be zero.
\end{proof}
Finally, we are in a position to state Moss's convergence theorem.
 \begin{thm}
Let $M, U,$ and $N$ be filtered objects of $\cat C$ with convergent spectral sequences such that $U$ is a ring, $M$ is a right $U$-module, and $N$ is a left $U$-module. Let $\langle V_0, \cdots , V_n \rangle$ be a multipliable system of matrices of permanent cycles the $E_k$ pages of the relevant spectral sequences, detecting matrices of elements $\langle \tilde V_0, \cdots \tilde V_n \rangle$ in the homotopy groups of the colimits of $M$, $U$, and $N$. Moreover, suppose that the spectral sequences for $M$, $U$, and $N$ satisfy the crossing differential hypothesis in the degrees of any product of elements from any number of adjacent matrices. 

Then, if the matric Massey product $\langle V_0, \cdots , V_n \rangle$ is not empty, then it contains a permanent cycle detecting an element of the smash Toda bracket $\langle \tilde V_0, \cdots \tilde V_n \rangle$
 \end{thm}
\begin{proof}
By Lemma \ref{survive}, the matrices $V_1, \cdots, V_{n-1}$ lift from $\pi^{k-1}_{**}(U \otimes C\tau^k)$ to matrices $\hat V_1, \cdots, \hat V_{n-1}$ in $\pi^{k-1}_{**}(U)$ with the property that $\tau^{-1}\hat V_i = \tilde V_i.$ Similarly, $V_0$ and $V_n$ lift to $\pi^{k-1}_{**}(M)$ and $\pi^{k-1}_{**}(N).$  The crossing differentials hypothesis then implies we can lift the entire chain complex $\mathbb{T}(V)$ of Definition \ref{smashtoda} from $\pi^{k-1}_{**}(U \otimes C\tau^k)$ to $\pi^{k-1}_{**}(U).$ Since the ``tensor with $C\tau^k$'' endofunctor of $\cat D^{k-1}$ is exact, it preserves the functors $C_0^{n-1}$ and $Z_0^{n-1}$, which can be written as iterated (co)fibers. Therefore the entire diagram
\[ 
\xymatrix@=2.5cm{
S^{\ell-1} \otimes C\tau^k \ar[r]^-{\Sigma^{\ell-1} V_1'} & \Sigma^{\ell -1} Z_0^{\ell-1}(M \otimes_U \mathbb{T}(V)) \ar[d]^\cong & \\
& C_0^{\ell-1}(M \otimes_U \mathbb{T}(V)) \ar[r]^{M \otimes_U V_{\ell+1}'}  &\Sigma^{d^\ell+1} M \otimes_U N
}
\]
lifts from $\Mod_{C\tau^k}(\cat D^{k-1}(\cat C))$ to $\cat D^{k-1}(\cat C),$ so we obtain an element $\alpha_V$ of $\langle \hat V_1, \cdots, \hat V_n \rangle$ which maps to an element $\alpha_{\hat{V}}$ of $\langle V_1, \cdots, V_n \rangle$. The statement then follows because the map 
\[
\tau^{-1}: \cat D^{k-1}(\cat C) \to \cat C
\]
is pointed and symmetric monoidal, and hence preserves smash Toda brackets, so we have \[
\tau^{-1} \langle \hat V_1, \cdots, \hat V_n \rangle \subseteq \langle \tilde V_1, \cdots, \tilde V_n \rangle.
\]
But $\tau^{-1}\alpha_{\hat{V}}$ is the set of elements detected by $\alpha_V,$ completing the proof.
\end{proof}
\section{Generation by Matric Massey Products}\label{gmsection}
\subsection{Introduction}
Recall that Gugenheim and May \cite{gm} proved that for any field $R$ and augmented connected graded algebra $A$, the cohomology $\Ext_A(R,R)$
of $A$ is generated by elements in degree $s = 1$ under matric Massey products. 

Much of Gugenheim and May's proof is formal, but there is a single step (Proposition 5.16) which relies on the explicit structure of the bar construction used to calculate $\Ext_A(R,R).$ This piece of the argument does not generalize to other DGAs, but we find that it can be replaced with an argument that works for differential graded algebra whose underlying algebra is \emph{Koszul}: that is, generated by elements in degree $s = 1$, with relations generated by elements in degree $s = 2$, with relations between relations generated in degree $s = 3$, and so on. In particular, this is true is the algebra is freely generated in degree $s = 1$, as in the cobar construction computing $\Ext_A(R,R).$

The main result of the chapter is the following:
\begin{thm}\label{mainmmpthm}
Let $E_r^{s,t}$ be a multiplicative spectral sequence over a field $R$ concentrated in degrees $s \geq 0$, and suppose the $E_1$ page can be chosen to be Koszul. Let $E_{2,r}^{s,t}$ denote the set of elements in $E_2^{s,t}$ which survive to the $E_r$ page. Then:
\begin{itemize}
\item The $E_2$ page is generated under matric Massey products by elements in degree $s = 1$. 
\item $E_{2,r}^{*,*}$ is generated under matric Massey products by elements in degree $0 < s < r.$
\end{itemize}
\end{thm}
The first part is a generalization of the main result of \cite[Chapter 5]{gm}, while the second is wholly new. The result holds in our motivating example of the Adams Spectral Sequence, but also more generally. 

We prove the first bullet point in Section \ref{gmkoszulsection}, and the second in Section \ref{filkoszulsection}. Before that, in Section \ref{classickossec}, we record a handful of classical algebraic facts that we will need in what follows.

\subsection{A classical identification of an augmented algebra's indecomposables}\label{classickossec}
This brief section contains classical algebraic facts and definition we will reference in the following two sections. First, we formalize our notion of a \emph{Koszul algebra}.
\begin{defn}
Let $R$ be a field, and let $U$ be a negatively-graded augmented graded $R$-algebra. We say $U$ is \emph{Koszul} if $\Tor^n_U(R,R)$ is concentrated in degree $-n.$
\end{defn}
Following Wall's work \cite{wall1960generators}, note that this places strong conditions on the generators of $U$. By writing down a minimal $U$-resolution of $R$, one observes that $U$ must be generated by elements in degree $n=-1$, with relations in degree $n=-2$, with relations between relations in degree $n=-3$, and so on.

Koszulity turns out to be the right condition to generalize Gugenheim and May's argument. For our applications, the most important example is:
\begin{ex}
Let $U$ be a free algebra generated by elements in degree $n = -1.$ Then $U$ is Koszul.
\end{ex}
Classically, we can identify indecomposables in $U$ with a subspace of $\Tor_1^U(R,R).$ The next two sections will use Gugenheim and May's insight that if $U$ carries a differential structure this identification can be taken further, to relate elements indecomposable \emph{by matric Massey products} in $U$ with a subspace of a \emph{derived} version of $\Tor_1^U(R,R).$ Here we record the classical statement and proof to motivate our proofs in the next two sections. To start, we define a map traditionally called the \emph{suspension}.
\begin{defn}\label{sigmaclassical}
Let $R$ be a ring and $U$ be an augmented $R$-algebra with augmentation ideal $IU.$ We define the \emph{suspension map}
\[
\sigma: IU \to \Tor_1^U(R,R)
\]
as follows. Choose a free $U$-resolution $X_\bullet$ of $R$ with $X_0 = U.$ 

For $u \in IU$, let $x$ denote the image of $u$ under the inclusion $IU \to X_0.$ Since $x = 0 $ in $H_0(X_\bullet)$, there exists $y$ in $X_1$ with $dy = x.$ Notice that $y \otimes 1$ is a cycle in $X_\bullet \otimes R$, and so corresponds to an element of $\Tor^U_1(R,R)$, which we call $\sigma(u).$ 

To see that this did not depend on the choice of $y,$ suppose that we instead chose $y'$ with $dy' = x$. Then $d(y - y') = 0,$ so there exists $z$ in $X_2$ with $dz = (y - y').$ In particular, $d(z \otimes 1) = (y \otimes 1 - y' \otimes 1)$ in $X_\bullet \otimes R,$ so $\sigma(u)$ is indeed well-defined in $\Tor^U_1(R,R).$ With further work it is possible to show that $\sigma$ did not depend on the choice of resolution, but this is not relevant to the rest of this paper, so we omit it.
\end{defn}

Notice that $\sigma$, as defined, is actually a map of $U$-modules, where we view the right hand side as the direct sum of a bunch of copies of $R$ with $U$-action induced by the augmentation map. In particular, we see that $\sigma(u) = 0$ whenever $u$ is the product of two elements of $IU.$ The converse is true as well. To see this, we require the following (immediate) observation. 
\begin{obs}\label{iuiuclassical}
Let $IU$ be the augmentation ideal of $U$, defined by the short exact sequence
\[
0 \to IU \to U \to R \to 0.
\]
An element of $IU$ is decomposable if and only if it lands in the image of the natural map
\[
IU \otimes_U IU \to IU.
\]
\end{obs}
This observation allows us to identify the kernel of $\sigma$ as the decomposables in $IU$, which will imply the desired bijection between the image of $\sigma$ and the indecomposables.
\begin{prop}\label{kersigmaclassical}
The kernel of $\sigma$ consists of precisely the decomposable elements of $IU$. 
\end{prop}
\begin{proof}
Let $\pi$ denote the natural map $IU = U \otimes_U IU \to R \otimes_U IU.$ The long exact sequence
\[
0 \to IU \to U \to R \to 0
\]
induces an exact sequence
\[
IU \otimes_U IU \to U \otimes_U IU \xrightarrow{\pi} R \otimes_U IU \to 0.
\]
so the kernel of $\pi$ is the image of $IU \otimes_U IU$ in $U \otimes_U IU,$ which by Observation \ref{iuiuclassical} consists of the decomposable elements of $IU$.

Now, consider the diagram
\[
\xymatrix{
 \Tor^U_1(R,U) \ar[r] &  \Tor^U_1(R,R)\ar[r] & R\otimes_U IU  \\
& IU \ar^\sigma[u] \ar^\pi[ur], &&&
 }
\]
which commutes by the definitions of $\pi$ and $\tau.$ The top row is exact and $\Tor^U_1(R,U)$ is 0, so $\ker \sigma$ must be equal to $\ker \pi$.
\end{proof}
\subsection{Koszul conditions and generation in Gugenheim-May}\label{gmkoszulsection}
Gugenheim and May study matric Massey products in $\Ext_A(R,R)$, which is the cohomology of the cobar construction of an augmented connected algebra $A$. In this section, we show that their argument generalizes to relate more general differential-graded algebras to their cohomology rings.

\begin{warn}\label{sngrading}
The $(n,t)$ gradings we have introduced are inherently \emph{homological}. Gugenheim and May use cohomological grading $(s,t)$, where $s = -n.$ We will find $n$-grading more convenient for our proofs, but state our final result with $s$-grading to facilitate comparison and application to the Adams spectral sequence.
\end{warn}

Our goal is to prove the following theorem, which is equivalent to the first bullet point of Theorem \ref{mainmmpthm}. The proof will closely follow Chapter 5 of Gugenheim and May's book \cite{gm}, which proves the case where $U$ is the cobar construction of an augmented connected algebra $A$. 
\begin{thm}
Let $U$ be a differential graded algebra whose underlying algebra is Koszul. The homology $H(U)$ is generated under matric Massey products by elements of degree $n = -1.$
\end{thm}
Our proof strategy is as follows. First, we will define a ``derived'' Koszul condition for differential graded algebras, which we will show implies that the homology is generated under matric Massey products by elements of degree $n = -1.$ Then, we will show that any differential graded algebra whose underlying algebra is Koszul (in the classical sense) is also Koszul in the derived sense. To start, we define:
\begin{defn}
Let $R$ be a field, and let $U$ be an augmented differential graded $R$-algebra. We say $U$ is \emph{Koszul} if $R \otimes_U^\LL R$ has homology concentrated in degree $n = 0.$
\end{defn}
Here $R \otimes_U^\LL R$ is the \emph{derived} tensor product, which Gugenheim and May notate as $\Tor^*_U(R,R).$ We state the two parts of our proof as separate theorems. 
\begin{thm}\label{koszulgendga}
Let $U$ be a Koszul differential graded $R$-algebra. Then $HU$ is generated by elements of degree $-1$ under matric Massey products.
\end{thm}
\begin{thm}\label{dgakoszuliskoszul}
Let $U$ be a differential graded $R$-algebra. If the underlying graded $R$-algebra of $U$ is Koszul, then $U$ is as well.
\end{thm}
The proof of Theorem \ref{koszulgendga} is essentially contained in Section 5 of \cite{gm}, although they do not state it in this generality. We sketch it here to emphasize that the proof works for any Koszul differential graded $R$-algebra, rather than just the cobar construction, and to foreshadow the proofs in the next section. We start by generalizing the suspension map from definition \ref{sigmaclassical} to include the presence of differentials. 

\begin{defn}[Definition 3.7 of \cite{gm}]
Let $R$ be a ring and $U$ be an augmented differential graded $R$-algebra with augmentation ideal $IU.$ Denote by $E^k_{**}(R,U,R)$ the algebraic Eilenberg-Moore spectral sequence computing the homology of $R \otimes_U^\LL R.$ The suspension map
\[
\sigma: H(IU) \to E^\infty_{1*}(R,U,R) 
\]
is the top horizontal composition
\[
\xymatrix{
H(IU) \ar^\sigma[r] \ar[d]^{\sigma} & E^\infty_{1*}(R,U,R) \\
\Tor_{1*}^{HU}(R,R) \ar@{=}[r] & E^2_{1*}(R,U,R) \ar[u],
}
\]
where the $\sigma$ appearing on the left hand arrow is the classical map of Definition \ref{sigmaclassical}.
\end{defn}
Gugenheim and May then prove a generalization of Proposition \ref{kersigmaclassical}:
\begin{thm}[Corollary 5.13 of \cite{gm}]
The kernel of the map $\sigma$ consists of the elements of $U$ which are decomposable by matric Massey products. 
\end{thm}
\begin{proof}[Sketch of proof]
As a derived counterpart of Observation \ref{iuiuclassical}, Gugenheim and May show that the image of the map
\[
H\left(IU \otimes_U^\LL IU\right)  \to H\left(IU\right)
\]
out of the derived tensor product is the set of decomposable in $H(IU)$ under \emph{matric Massey products.} The fiber sequence
\[
IU \to U \to R
\]
induces a fiber sequence
\[
IU \otimes_U^\LL IU  \to IU \otimes^\LL_U U \xrightarrow{\pi} IU \otimes_U^\LL R,
\]
which then leads to an exact sequence
\[
H\left(IU \otimes_U^\LL IU\right)  \to H\left(IU\right) \xrightarrow{\pi} H\left(R \otimes_U^\LL IU\right).
\]
In particular, the kernel of the map $\pi$ must be the set of indecomposables. Following the proof of the previous section, one sees that $\ker \pi  = \ker \sigma.$ 
\end{proof}

The proof of Theorem \ref{koszulgendga} is then immediate. The Koszul condition implies that $E^\infty_{1,n}(R,U,R)$ is concentrated in degree $-1$, so elements in any other degree must land in the kernel of $\sigma$ and hence be decomposable. 

In the remainder of this section, we prove Theorem \ref{dgakoszuliskoszul}. The argument is a variant of the proof of Proposition 5.16 in \cite{gm} designed to avoid explicit features of the cobar construction. To start, we place a filtration on an arbitrary chain complex $X$.
\begin{defn}
Let $X$ be a chain complex. The \emph{inverse filtration} on $X$ is the natural filtration with 
$$F_n(X)_i = \begin{cases}
    X_i & \text{  for  } i \leq - n \\
    0 & \text{  for  } i > -n.
\end{cases}$$ 
\end{defn}
We call this the ``inverse'' filtration to contrast it with the canonical filtration
$$F_n(X)^{\text{can}}_i = \begin{cases}
    X_i & \text{  for  } i > n \\
    \ker(d_i) & \text{  for  } i = n \\
    0 & \text{  for  } < n.
\end{cases}$$ 
which runs in the opposite direction. While the canonical filtration leads to a spectral sequence based on the homology of the DGA (traditionally called the algebraic Eilenberg-Moore spectral sequence), the inverse filtration leads to a spectral sequence based on the DGA's underlying algebra.
\begin{proof}[Proof of Theorem \ref{dgakoszuliskoszul}]
Consider the inverse filtration on $U$, and note that the associated graded complex of $F_\bullet(U)$ is equivalent to the underlying algebra $V(U)$  of $U$ with zero differential. Applying \ref{torfil} to with filtration on $U$ and the trivial filtration $S\iota R$ on $R$, we obtain a spectral sequence
\[
E^1_{**} = \Tor_*^{V(U)}(R,R) \Rightarrow H(R \otimes^\LL_U R).
\]
Since $V(U)$ is Koszul, the $E^1$ page is concentrated in degrees $(n,-n)$, so $ H(R \otimes^\LL_U R)$ must be concentrated in degree zero.
\end{proof}
\subsection{Koszul conditions and generation in $\Fil(\cat C).$}\label{filkoszulsection}
In this section we prove the second bullet point of Theorem \ref{mainmmpthm}.
In particular, we prove:
\begin{thm}\label{secondbullet}
Let $E_r^{s,t}$ be a multiplicative spectral sequence over a field $R$, and suppose the $E_1$ page can be chosen to be freely generated by elements in degree $s = 1$ in addition to a unital copy of $R$ in degree $s = t = 0.$  Let $E_{2,r}^{s,t}$ denote the set of elements in $E_2^{s,t}$ which survive to the $E_r$ page. Then:
\begin{itemize}
\item $E_{2,r}^{*,*}$ is generated under matric Massey products by elements in degree $0 < s < r.$
\end{itemize}
\end{thm}

Our strategy is similar to that of the previous section. First, applying Lemma \ref{multssarefil} if necessary, we may assume our spectral spectral sequence comes from a ring object in $\Fil(\cat C)$ for some stable $\infty$-category $\cat C.$ Next, we define another ``derived'' variant of the Koszul condition, which is essentially the definition of the previous section constructed within the special fiber of the $\cat D^1$ deformation. We show that this condition implies the desired generation condition on $E_{2,r}^{s,t}$, and then we show that a Koszul $E^1$ page makes the spectral sequence Koszul in this derived sense. 

Throughout this section, we use the notation $\otimes^\LL$ to emphasize that our tensor products are taken in $\cat D^1(\cat C)$ rather than in $\Fil(\cat C)$, and we use $n$-grading, directing the reader who prefers $s$-grading to Warning \ref{sngrading}. Our Koszul condition is as follows.
\begin{defn}
Let $R$ be a(n ordinary) field viewed as an object of $\cat C$, and let $U$ be an augmented $R$-algebra in $\cat D^1(\cat C).$ We say $U$ is \emph{Koszul} if $\pi_{nt}^1(R\otimes^\LL_U R \otimes^\LL C\tau^2)
$
is concentrated in degree $0.$
\end{defn}
The goal of this section is to prove two theorems:
\begin{thm}\label{koszulgenfil}
Let $U$ be a Koszul $R$-algebra which is cofibrant in the $1$-projective model structure on $\cat Fil(\cat C)$. Then for each $r$, $\pi^1_{nt}(U \otimes C\tau^{r+1})$ is generated under (matric) smash Toda brackets by elements with degree $1-r \geq n \geq -1.$ 
\end{thm}
\begin{thm}\label{filkoszuliskoszul}
Suppose $U$ is a ring in $\Fil(\cat C)$ which is cofibrant in the 1-projective model structure, and $\pi^0_{**}(U \otimes C\tau)$ is Koszul. Then $U$ is as well. 
\end{thm}
These will jointly imply Theorem \ref{secondbullet}, because by Lemma \ref{survive}, $E^{2,r}_{n,t}(U)$ is the image of the map
\[
\pi^1_{nt}(U \otimes C\tau^{r+1}) \to \pi^1_{nt}(U \otimes C\tau^{2}),
\]
and this map preserves (matric) smash Toda brackets.

We begin with the following lemma.
\begin{lem}
We have an identification 
\[
R \otimes^\LL_U R \otimes^\LL C\tau^{k+1} \to (R \otimes^\LL C\tau^{k+1}) \otimes^\LL_{U \otimes^\LL C\tau^{k+1}} (R \otimes^\LL C\tau^{k+1}).
\]
\end{lem}
\begin{proof}
This is the same argument as the proof of Lemma \ref{torfil}.
\end{proof}
We now continue as before by defining a suspension map. 
\begin{defn}
Let $U$ be an augmented $R$-module in $\cat D^1(\cat C)$ with augmentation ideal $IU$, and let $E^*(U)$ denote the spectral sequence from Theorem \ref{torclassic} with
\[
E^2_{rnt}(U) = \Tor_r^{\pi^1_{**}\left(U\otimes^\LL C\tau^{k+1}\right)}\left(\pi^1_{**}(R\otimes^\LL C\tau^{k+1}),\pi^1_{**}(R\otimes^\LL C\tau^{k+1})\right)
\]
converging to
\[
\pi^1_{n-r,t+r}\left(R \otimes^\LL_{U} R\otimes^\LL C\tau^{k+1}\right).
\]
We define
\[
\sigma: \pi^1_{**}(IU\otimes^\LL C\tau^{k+1}) \to E^\infty_{1**}(U) 
\]
as the top horizontal composition
\[
\xymatrix{
\pi^1_{**}(IU\otimes^\LL C\tau^{k+1}) \ar^\sigma[r] \ar[d]^{\sigma} & E^\infty_{1**}(U) \\
\Tor_{1}^{\pi^1_{**}(IU\otimes^\LL C\tau^{k+1})}(\pi^1_{**}(R\otimes^\LL C\tau^{k+1}),\pi^1_{**}(R\otimes^\LL C\tau^{k+1})) \ar@{=}[r] & E^2_{1**}(U) \ar[u],
}
\]
where the $\sigma$ labelling the left hand arrow is the classical suspension map from Definition \ref{sigmaclassical}.
\end{defn}
As before, the first step to relating $\sigma$ with decomposables is to identify indecomposables with a subspace of the homotopy groups of a tensor product. In this setting, we have the following version of Observation \ref{iuiuclassical}.
\begin{obs}\label{iuiufil}
Let $U$ be a (not-necessarily commutative) $R$-algebra in $\cat D^1(\cat C)$ with an augmentation giving rise to a split fiber sequence
\[
IU \to U \to R,
\]
and fix a right $U$-module $M$. Every element in the image of
\[
\pi^1_{**}(M \otimes^\LL_U IU) \to \pi^1_{**}(M \otimes^\LL_U U) 
\]
is decomposable by nontrivial smash Toda brackets (by which we mean smash Toda brackets such that the last factor comes from $\pi^1_{**}(IU),$ not just $\pi^1_{**}(U).$)
\end{obs}
\begin{proof}
Fix a cofibrant replacement $X_\bullet$ of $IU$, taken in the 1-projective model structure on $\Fil(\Mod_U).$ Combining the results of section 3 with Theorem \ref{ariottathm}, we can view this resolution as a chain complex $\Ch(X)_\bullet$. Note that we can build such a replacement from any free $\pi^1_{**}(U)$ resolution of $\pi^1_{**}(IU)$, so in particular we can assume the elements of $X_\bullet \otimes C\tau$ are coproducts of sums of shifts of $U$, not merely of retracts of $U$. 

Now, choose $\alpha \in \pi^1_{**}(M\otimes^\LL_U IU)$. Since $\colim X_\bullet = IU$ and spheres are compact by Theorem \ref{compactspheres}, there must be some $n> 0$ such that $\alpha$ lifts to some $\hat\alpha$ in $\pi^1_{n't}\left(M \otimes^\LL X_{-n} \right).$  But by Theorem \ref{ariottathm}, we can identify $X_{-n}$ with $\Sigma^nZ_0^n(\Ch(X)_\bullet)$, which by Proposition \ref{iter} is equivalent to $C_0^n(\Ch(X)_\bullet)).$ Since $\Ch(X_\bullet)$ is a complex of direct sums of shifts of $U$, the maps $\Ch(X)_k \to \Ch(X)_{k-1}$ are matrix maps with elements in $\pi^1_{**}(U)$. Call these matrices $V_2,\cdots, V_{n}$. Then, we can express $\alpha$ as a composition
\[
\xymatrix{
S^{n',t} \ar[r] & \Sigma^{\ell -1} Z_0^{\ell-1}(M \otimes_U^\LL X|_{\Ch_0^n}) \ar[d]^\cong & \\
& C_0^{\ell-1}(M \otimes_U^\LL X|_{\Ch_0^n}) \ar[r]  & M \otimes_U IU
}
\]
and therefore (using Definition \ref{smashtoda}) as an element of a nontrivial smash Toda bracket. Composing this with the map $M \otimes_U IU \to M \otimes_U U$ exhibits the image of $\alpha$ as a nontrivial smash Toda bracket, as desired.
\end{proof}
With this observation, the counterpart of Proposition \ref{kersigmaclassical} is very similar to the classical case. 
\begin{prop}\label{kersigma}
The kernel of $\sigma$ consists of elements decomposable in terms of matric Toda brackets.
\end{prop}
\begin{proof}
Let $\pi$ denote the composition 
\[
\xymatrix{
\pi^1_{**}(IU\otimes^\LL C\tau^{k+1}) \ar@{=}[r] \ar^\pi[rd] & \pi^1_{**}(IU \otimes^\LL C\tau^{k+1} \otimes^\LL_{U\otimes^\LL C\tau^{k+1}} U\otimes^\LL C\tau^{k+1}) \ar[d] \\& \pi^1_{**}(IU \otimes^\LL C\tau^{k+1}\otimes^\LL_{U\otimes^\LL C\tau^{k+1}} R\otimes^\LL C\tau^{k+1})}
\]
induced by the augmentation map.

The kernel of $\pi$ is the image of $$\pi^1_{**}(IU \otimes^\LL C\tau^{k+1}\otimes^\LL_{U\otimes^\LL C\tau^{k+1}} IU\otimes^\LL C\tau^{k+1})$$ in $\pi^1_{**}(IU\otimes^\LL C\tau^{k+1})$, which by Observation \ref{iuiufil} is the collection of nontrivial matric Toda brackets. Now, note that $E^\infty_{0nt} = R[\tau]/\tau^{k+1}$ is a direct summand of $\pi^1_{**}(R \otimes^\LL_U R\otimes C\tau^{k+1})$, so we may extend $\sigma$ to a map 
\[
\xymatrix{
E^\infty_{1nt}(U) \ar[r] & \pi^1_{n-1, t+1}(R \otimes^\LL C\tau^{k+1}\otimes^\LL_{U\otimes^\LL C\tau^{k+1}} R\otimes^\LL C\tau^{k+1}) \\
\pi^1_{nt}(IU\otimes^\LL C\tau^{k+1}) \ar^\sigma[u] \ar^{\sigma'}[ru] &
}
\]
Now, since $\pi^1_{**}(R\otimes^\LL C\tau^{k+1}) = R[\tau]/\tau^k$ is a direct summand of $\pi^1_{**}(R \otimes^\LL_U R)$, we obtain a diagram whose rightmost column is exact.
\[
\xymatrix{
&0 \ar[d] \\
&\pi^1_{nt}(R \otimes^\LL_U U\otimes C\tau^{k+1})\ar[d] \\
 \pi^1_{n+1,t-1}(IU\otimes^\LL C\tau^{k+1}) \ar[r]^\sigma \ar[rd]^\pi & \pi^1_{n+1,t-1}(R \otimes^\LL_U R\otimes^\LL C\tau^{k+1}) \ar[d] \\
 & \pi^1_{n+1,t-1}(R \otimes^\LL_U IU\otimes^\LL C\tau^{k+1})
}
\]
Comparing the definitions of $\pi$ and $\sigma$ shows that the diagram commutes. Moreover, $0$ is the only element in both the image of $\sigma$ and the image of $R[\tau]/\tau^k$, so it follows that $\ker \sigma = \ker \pi.$
\end{proof}
To complete the proof, we need to constrain $\pi_1(R \otimes^\LL_U R \otimes^\LL C\tau^{k+1}).$

\begin{lem}\label{ctkp1degrees}
Suppose $\pi^1_{**}(X \otimes^\LL C\tau^2)$ is concentrated in degree $n = 0$. Then $\pi^1_{**}(X \otimes^\LL C\tau^{k+1})$ is concentrated in degrees between $0$ and $1-k.$ 
\end{lem}
\begin{proof}
Filter $C\tau^{k+1}$ as
\[
\Sigma^{-k,0} C\tau \to \Sigma^{1-k,0} C\tau^2 \to \cdots \to C\tau^k.
\]
Using Corollary \ref{kcofiber} and Lemma \ref{ctct}, we can compute the cofiber in $\cat D^1(\cat C)$ of the map
\[
\Sigma^{-k + j -1, 0} C\tau^{j} \to \Sigma^{-k+j,0} C\tau^{j+1}
\]
to be $\Sigma^{-k+j,0}C\tau^2 \oplus \Sigma^{-k-1,1} C\tau$, which is weak equivalent to $\Sigma^{-k+j,0} C\tau^2.$

Applying Theorem \ref{torfil} to this filtration, we obtain a spectral sequence converging to $\pi^1_{**}(X \otimes^\LL C\tau^{k+1})$ with
\[
E^2_{rnt} = \pi^1_{nt}(X \otimes^\LL \bigoplus_{r = 1}^{k} \Sigma^{-k+r,0} C\tau^2),
\]
which is concentrated in degrees $1-k \leq n \leq 0.$
\end{proof}

The proof of Theorem \ref{koszulgenfil} is then immediate. Proposition \ref{kersigma} implies that any element of $\pi^1_{**}(IU \otimes^\LL C\tau^{k+1}$ outside the kernel of sigma is a nontrivial smash Toda bracket. But Lemma \ref{ctkp1degrees} and the Koszul condition imply that the kernel is trivial outside of degrees $1-k \leq n \leq 0,$ completing the proof.

Now, we prove Theorem \ref{filkoszuliskoszul}. 

\begin{defn}
Let $U$ be an object of $\Fil(\cat C).$ The \emph{inverse filtration} on $U$ is the natural object $F_\bullet(U)$ in $\Fil(\Fil(\cat C))$ with
\[
\left(F_n(U)\right)_i = \begin{cases}
    U_i & \text{  for  } i > n\\
    U_n & \text{  for  } i \leq n.
\end{cases}
\]
\end{defn}
Note that if $U$ is cofibrant, the map $F_n(U) \to F_{n-1}(U)$ is a cofibration, so this is also a filtration of $U$ in $\cat D^1(\cat C)$.

Moreover, in this case we can use Corollary \ref{kcofiber} to compute the homotopy cofiber of the map $F_{n+1}(U) \to F_n(U)$ in the 1-projective model structure. The cofiber takes the form
\[
\cdots \to X_{n+2}/X_{n+3} \xrightarrow{0} X_{n+1}/X_{n+2} \xrightarrow{0} X_n/X_{n+1} \xrightarrow{id} X_n/X_{n+1} \xrightarrow{id} X_n/X_{n+1} \to \cdots 
\]
which is weak equivalent to 
\[
\cdots \to 0 \to 0 \to  X_n/X_{n+1} \xrightarrow{id} X_n/X_{n+1} \xrightarrow{id} X_n/X_{n+1} \to \cdots 
\]
which is precisely the cofiber of this map computed in $\Fil(\cat C).$ This observation will help us in the following proof:
\begin{proof}[Proof of Theorem \ref{filkoszuliskoszul}]
We want to understand $\pi^1_{**}(R \otimes^\LL_U R \otimes C\tau^2).$ As before, this is equivalent to
\[
(R \otimes C\tau^2) \otimes^\LL_{U \otimes C\tau^2} (R \otimes C\tau^2).
\]
Placing the inverse filtration on $U$, we find that the associated graded complex $G_n(U)$ takes the form
\begin{align*}
G_n(U) &= \left(F_n(U) \otimes C\tau^2\right)/\left(F_{n+1}(U) \otimes C\tau^2\right)\\
&= \left(F_n(U)/F_{n+1}(U)\right) \otimes C\tau^2 \\
            &= \Sigma^{r,0} C\tau^2 \otimes U_r/U_{r+1}.
\end{align*}
By Theorem \ref{torfil}, there is a spectral sequence  with
\[
E^1_{n'nt} = \pi^0_{n'nt}\left(R \otimes C\tau^2 \otimes^\LL_{G_n(U)} R \otimes C\tau^2 \right).
\]
Here the right hand side is being computed in $\Fil(\cat D^1(\cat C)),$ with the index $n'$ coming from the $\Fil$-filtration and $n$ coming from the $\cat D^1$-filtration. The sequence converges to $\pi^1_{n,t}(R \otimes^\LL_U R \otimes C\tau^2).$ 

To compute the $E^1$ page, we apply the spectral sequence of Theorem \ref{torclassic}, which takes the form
\[
E^2_{rn'nt}= \Tor^r_{\pi_{**}^0(U)}(R,R)
\]
 and converges to
\[
\pi^0_{n' + r, n + r, t-r}\left(R \otimes C\tau^2 \otimes^\LL_{G_n(U)} R \otimes C\tau^2 \right).
\]
Looking closely at the filtration $G_n(U)$, we note that both $n$ and $n'$ are equal to the internal $n$-grading on the homotopy groups/rings on the right hand side. By our Koszulity assumption, this $E^2$ page is concentrated in degrees $r = -n.$ In particular, we conclude that $\pi^1_{**}(R \otimes^\LL_U R \otimes C\tau^2)$ and so $U$ is Koszul. 
\end{proof}

\section{Additional lemmas}
In this appendix we gather a handful of lemmas used at various points in the text.
\begin{lem}\label{repcomp}
Let $A_n \to \cdots \to A_0$ and $B_n \to \cdots \to B_0$ be sequences of maps (not chain complexes) of sets and suppose:
\begin{enumerate}
    \item $f$ maps the image of $A_n$ in $A_1$ isomorphically onto the image of $B_n$ in $B_1$.
    \item  $f$ maps the image of $A_{n-1}$ in $A_0$ isomorphically onto the image of $B_{n-1}$ in $B_0$.
\end{enumerate}
Then, $f$ maps the image of $A_n$ in $A_0$ isomorphically onto the image of $B_n$ in $B_0.$
\end{lem}
\begin{proof}
Call the map from the image of $A_n$ in $A_0$ to the image of $B_n$ in $B_0$ $f_*.$ The image of $A_n$ in $A_0$ is a subset of the image of $A_{n-1}$ in $A_0$ (and the same for $B$), so $f_*$ is injective. Similarly, the image of $A_n$ in $A_0$ is a quotient of the image of $A_n$ in $A_1$, so $f_*$ is surjective. 
\end{proof}

 \begin{lem}\label{acycfib}
Suppose $f: X \to Y$ is $k$-exact and a $J^k$ injection. Then, the fiber $F$ of $f$ is $k$-acyclic.
 \end{lem}
\begin{proof}
Consider the following diagram, noting that the vertical triples are fiber sequences.
\[
\xymatrix{
       &  & F_{n-k}/F_n \ar[r]^{\delta} \ar[d]^{i} & \Sigma F_n \ar[d]^{i} \\
\Omega X_{n-2k}/X_{n-k} \ar[r]^\delta \ar[d]^{f} & X_{n-k} \ar[d]^{f} \ar[r]^\ell & X_{n-k}/X_n \ar[r]^\delta \ar[d]^{f} & \Sigma X_n \ar[d]^{f} \\
\Omega Y_{n-2k}/Y_{n-k}\ar[r]^\delta &Y_{n-k} \ar[r]^\ell &  Y_{n-k}/Y_n\ar[r]^\delta  & \Sigma Y_n \\
}
\]
 Pick an arbitrary $a$ in $\pi_{t}(\Sigma F_n).$ To show that $\tau^k a = 0$, we need to show that $a$ lifts to $\pi_t(F_{n-k}/F_n).$  Since $f_*(i_*(a)) = 0$ and $f$ is exact, there exists $b$ in $\pi_t(X_{n-k}/X_n)$ with $\delta_*b = i_*(a).$ 

Now, $f_*(b)$ might be nonzero. But $\delta_*f_*(b) = f_*(\delta_*b) = 0,$ so there exists $c$ in $\pi_t(Y_{n-k})$ such that $\ell_*(c) = f_*(b).$ Now, since $f$ is exact, there exist $y$ in  $\pi_t(\Omega Y_{n-2k}/Y_{n-k})$ and $s$ in $\pi_t(\Omega X_{n-2k}/X_{n-k})$ such that $f_*(s) = c + d_*y.$ Since $f$ is a $J^k$ injection, there exists $x$ in $\pi_t(\Omega X_{n-2k}/X_{n-k})$ with $f_*(x) = y,$ so in particular $f_*(s - d_*x) = c. $

Then, $f_*(\ell_*(s - d_*x)) = \ell_*(c) = f_*(b).$ But by definition, $d \circ \ell = 0$, so $d_*\ell_*( s - d_*x) = 0.$ So if we set $t = b + \ell(dx - c),$ we have:
\begin{itemize}
    \item $f_*(t) = 0$ 
    \item $d_*(t) = i_*(a)$
\end{itemize}
This guarantees $t$ lifts to $u$ in $\pi_t(F_{n-k}/F_n)$ with $d_*u = a$, as desired!
 \end{proof}
 
\begin{lem}\label{deltader} 
The cofiber map $\delta$ defined by the cofiber sequence
\[
\Sigma^{-k,0} C\tau^k \to C\tau^{2k} \to C\tau^k \xrightarrow{\delta} \Sigma^{-k,1} C\tau^{k}
\]
is a derivation, in the sense that we have a commutative square
\[
\xymatrix{
C\tau^k \otimes C\tau^k \ar[r]^\mu \ar[d]^{1 \otimes \delta + \delta \otimes 1} & C\tau^k \ar[d]^{\delta} \\
C\tau^k \otimes \Sigma^{-k,1} C\tau^k \oplus \Sigma^{-k,1} C\tau^k \otimes C\tau^k \ar[r] & \Sigma^{-k,1} C\tau^k}
\]
where the bottom horizontal map is the sum of two copies of $\mu$.
\end{lem}
This lemma is essentially a standard step in proving that the spectral sequence associated to a ring in $\Fil(\cat C)$ is multiplicative: we prove the corresponding statement about the representing objects $C\tau^k$ rather than the spectral sequence itself.  Our proof closely follows appendix A.3 of \cite{balderrama2021deformations}. We prove the Lemma over $\FF_2$ to avoid signs, and refer the interested reader to the aforementioned appendix to fill in the correct signs. 
\begin{proof}

We have a commutative square
\[\xymatrix{
S^{-k,0} \otimes S^{-k,0} \ar[r] \ar[d] & S^{0,0} \otimes S^{-k,0} \ar[d] \\
S^{-k,0} \otimes S^{0,0}\ar[r] & S^{0,0} \otimes S^{0,0}
}
\]
with an evident isomorphism to the commutative square 
\[
\xymatrix{
S^{-2k,0} \ar[r] \ar[d] & S^{-k,0} \ar[d] \\
S^{-k,0} \ar[r] & S^{0,0}
}
\]
Let $PO$ denote the pushout 
\[\xymatrix{
S^{-k,0} \otimes S^{-k,0} \ar[r] \ar[d] & S^{0,0} \otimes S^{-k,0} \ar[d] \\
S^{-k,0} \otimes S^{0,0} \ar[r] & PO.
}
\]
By standard results on total cofibers, we obtain a cofiber sequence
\[
PO \to S^{0,0} \otimes S^{0,0} \to C\tau^k \otimes C\tau^k.
\]
Now, the commutative square 
\[\xymatrix{
PO \ar[r] \ar[d] & S^{0,0} \otimes S^{0,0} \ar[d] \\
S^{-k,0}  \ar[r] & S^{0,0}
}
\]
extends to a map of cofiber sequences
\[\xymatrix{
PO \ar[r] \ar[d] & S^{0,0} \otimes S^{0,0} \ar[d] \ar[r] & C\tau^k \otimes C\tau^k \ar[d] \ar[r] & \Sigma^{0,1} PO \ar[d] \\
S^{-k,0}  \ar[r] & S^{0,0} \ar[r] & C\tau^k \ar[r] & S^{-k,1}
}
\]
and a simple inspection shows that the given map is the multiplication map.  

From here, the commutative square
\[\xymatrix{
S^{-k,0} \otimes S^{-k,0} \ar[r] \ar[d] & PO  \ar[d] \\
S^{-2k,0}  \ar[r] & S^{-k,0}
}
\]
extends to a map of cofiber sequences
\[\xymatrix{
S^{-k,0} \otimes S^{-k,0} \ar[r] \ar[d] & PO \ar[r] \ar[d] & S^{-k,0} \otimes C\tau^k \oplus C\tau^k \otimes S^{-k,0} \ar[d] \\
S^{-2k,0}  \ar[r] & S^{-k,0} \ar[r] & \Sigma^{-k,0} C\tau^k
}
\]
Pasting the rightmost square of the first cofiber sequence to the rightmost square of the second cofiber sequence gives a commutative diagram
\[
\xymatrix{
C\tau^k \otimes C\tau^k \ar[rr] \ar[d] && C\tau^k \ar[d]\\
\Sigma^{0,1} PO \ar[rr] \ar[d] && \Sigma S^{-k}\ar[d] \\
  S^{-k,1} \otimes C\tau^k \oplus C\tau^k \otimes S^{-k,1} \ar[rr] \ar[dr] && \Sigma^{-k,1} C\tau^k \\
  &\Sigma^{-k,1} C\tau^k  \otimes C\tau^k \oplus C\tau^k \otimes \Sigma^{-k,1} C\tau^k \ar^{\Sigma^{-k,1} (\mu \oplus\mu )}[ru] &
}
\]
where the bottom triangle comes from the unit axiom of a ring. 

By inspection, the left vertical composition is $\delta \otimes 1 \oplus 1 \otimes \delta$, and the right vertical composition is $\delta$, giving us the desired square. 
\end{proof}

\begin{lem}\label{multssarefil}
Any bigraded multiplicative spectral sequence over a field starting whose unit is a permanent cycle is isomorphic,  starting at the $E^1$ page, to the spectral sequence associated to a filtered DGA.
\end{lem}
\begin{proof}
 We can decompose the $E_1$ page into a vector space $S_\infty$ of permanent cycles, a sequence of vector spaces $S_i$ of sources of nonzero differentials on the $E_i$ page, and a sequence of vector spaces $T_i$ of targets of nonzero differentials on the $E_i$ page. Note that we have a canonical isomorphism $S_i \to T_i.$

Now, let $X$ be the filtered differential graded $R$-module
\[
X = S_\infty \oplus \mathbb{D}_1 \otimes_{\FF_p} S_1 \oplus \mathbb{D}_2 \otimes_{\FF_p} S_2 \oplus \cdots
\]
where $S_\infty$ has filtration zero and
\[
\mathbb{D}_i = \cdots \to 0  \to R \to R \to 0 \to \cdots
\]
where the first $R$ has filtration $0$ and the second lives in filtration $i$. 

The spectral sequence associated to $X$ is by construction  (additively) isomorphic the one we started with. A multiplicative structure on our original spectral sequence places a ring structure on the $E_1$ page, which is the bigraded module obtained by forgetting the differential on $X.$ The Leibniz rule says precisely that this multiplication preserves the differential on $X$, so we obtain a DGA with the desired spectral sequence.
\end{proof}
 
 \section{Proof of Proposition \ref{iter}}\label{appiter}
 In this appendix, we provide a proof of Proposition \ref{iter}. Throughout, we use Lurie's notation $\cat D_{x/}$ to denote the overcategory of maps in a category $\cat D$ from an object $x$. We will prove the proposition at the end of this appendix by writing down a diagram for which both $C_i^j(X)$ and $\Sigma^{j-i} Z_i^j(X)$ are colimits. In order to do so, we need four lemmas letting us manipulate poset-shaped diagrams without changing their colimit.  The first three are special cases of the $\infty$-categorical version of Quillen's theorem A, which states that a functor $i: \cat E \to \cat D$ is cofinal if and only if for every object $x \in \cat D$, the comma $\infty$-category $x/i$ (defined to be $\cat D_{x/} \times_{\cat D} \cat E$) is a weakly contractible simplicial set.
\begin{lem}\label{sfin}
Let $\cat D$ be a partially ordered set, and let $\cat E$ be a subcategory of shape $a \leftarrow b \to c.$ In order that the inclusion $i: \cat E \to \cat D$ be final, it suffices to check that:
\begin{itemize}
    \item Every $d \in \cat D$ has a map to either $a$ or $c$. 
    \item $b$ is the (categorical) product of $a$ and $c$.
\end{itemize}
\end{lem}
\begin{proof}
For arbitrary $d \in \cat D$, the first condition implies that the comma category $d / i$ is a nonempty subcategory of $\cat E$, and the second implies that it cannot be $\{a,c\}.$ 
\end{proof}
\begin{lem}\label{qa}
Let $\cat D$ be an ordinary category, and $d$ be an object of $\cat D$ equipped with a morphism $f: d \to e$ such that every morphism out of $d$ factors as a composition $g \circ f.$ Then, the inclusion $i: \cat D - \{d\} \to \cat D$ is final.
\end{lem}
\begin{proof}
Pick an arbitrary $x \in \cat D$. If $x \neq d$, then the comma category $x/i$ is just $(\cat D - \{d\})_{x/}$. If $x = d$, then $x/i = (\cat D - \{d\})_{e/}$ by the factorization assumption. In either case, $x/i$ is weakly contractible.
\end{proof}

\begin{lem}\label{lift}
Let $\cat D$ be a partially ordered set with objects $a,b,c$ such that:
\begin{itemize}
    \item There are maps $a \to c$ and $b \to c$.
    \item There are no maps $a \to b$ or $b \to a$.
    \item For any map $d \to a$, there exists a coproduct $b \coprod d$ in $\cat D.$
\end{itemize}

Let $\cat E$ be the partially ordered set extending $\cat D$ to include a map $a \to b$

Then, the inclusion $i: \cat D \to \cat E$ is final. Moreover, any functor $\cat D \to \cat C$ sending $a$ and $b$ to zero objects extends to a functor $\cat E \to \cat C.$
\end{lem}
\begin{proof}
Consider an arbitrary element $d$ in $\cat D.$ 

If $d$ does not map to $a$, then the simplicial set $d/i$ is equivalent to $D_{d/},$ which is of course contractible.

If $d$ does map to $a$, then $d/i$ is the (not disjoint) union $D_{d/} \cup D_{b/}.$ The intersection is $D_{d \coprod b /},$ implying the union is weakly contractible.

The rest of the claim follows from the universal properties of zero objects.
\end{proof}
The fourth lemma is a special case of \cite{htt} 4.2.3.10, and will give conditions under which we can remove from a diagram objects which are already colimits of certain subdiagrams. 
\begin{lem}\label{colim}
Let $\cat D$ be a partially ordered set. For a given object $x$, let ${\cat D'}_{/x}$ denote the sub-$\infty$-category of $D$ consisting of objects with a map to $x$, excluding $x$ itself.

Given an object $d$ and a functor $F$ from $\cat D$ into an $\infty$-category $\cat C$  such that $F|_{\cat D_{/d}}$ exhibits $F(d)$ as the colimit of $F|_{\cat D'_{/d}}$, it follows that the colimit of $F$ is the same as the colimit of $F|_{\cat D - \{d\}}.$
\end{lem}
\begin{proof}
This is a special case of HTT 4.2.3.10. Let $K$ be the nerve of $\cat D - \{d\}$, and let $\cat J = \cat D.$ The functor $\cat D \to \sSet_{/K}$ is defined on objects as follows:
\begin{itemize}
    \item For $x \neq d$, send $x$ to $K_{/x}$.
    \item For $x = d$, send $x$ to $\cat D'_{/d}.$ 
\end{itemize}
Note that $\cat J_\sigma$ is automatically contractible, because for $\sigma = x_1 \to \cdots \to x_n$, $\cat J_\sigma$ is isomorphic to $\cat D_{x_n/},$ so the hypotheses of HTT 4.2.3.10 are satisfied. The lemma then follows because the induced diagram $q|_{N(\cat j)}$ is equivalent to $F$. 
\end{proof}

\begin{proof}[Proof of Propositions \ref{iter}]
Assume without loss of generality that $i = 0$, and set $j$ equal to $n$ so we can use $i$ and $j$ to index things throughout the proof. We will construct a diagram $F: \cat D \to \cat S$ whose colimit is $C_0^n$, and use the previous four lemmas to transform the diagram without changing the colimit, until we end up with a diagram whose colimit is $\Sigma^{n}Z_0^n.$ 

Construct a poset-shaped diagram $\cat D$ inductively as follows:
\begin{itemize}
    \item Start with the subdiagram $X_n \to X_{n-1} \to \cdots \to X_0$ of $X.$
    \item Attach the pullback (and hence pushout) square \[
\xymatrix{
Z_0^1 \ar[r] \ar[d] & 0^1 \ar[d] \\
X_1 \ar[r] &  X_0
},
\]
along the edge $X_1 \to X_0$, (where we label individual zero objects $0_i$ and $0^i$ to help keep track of them later.)
\item By the universal property, the map $X_2 \to X_1$ lifts to $Z_0^1,$ and we can attach the pushout square
\[
\xymatrix{
Z_0^2 \ar[r] \ar[d] & 0^2 \ar[d] \\
X_2 \ar[r] &  Z_0^1
}.
\]
\item Continue this pattern until we've attached the pushout square
\[
\xymatrix{
Z_0^n \ar[r] \ar[d] & 0^n \ar[d] \\
X_n \ar[r] &  Z_0^{n-1}
}.
\]
\item In the same way, we can attach the pushout square
\[
\xymatrix{
X_n \ar[r] \ar[d] & X_{n-1} \ar[d] \\
0_n \ar[r] &  C_{n-1}^n
}
\]
and factor $X_{n-1} \to X_{n-2}$ through a map $C_{n-1}^n \to X_1$. We continue as before until we've attached the pushout square
\[
\xymatrix{
C^n_1 \ar[r] \ar[d] & X_0 \ar[d] \\
0_1 \ar[r] &  C_0^n
}
\]
\end{itemize}
Let $\cat D \to \cat S$ be the diagram constructed above, \emph{excluding} $C_0^n.$  To give some idea what we mean, here's $\cat D$ when $n = 3$: 
\[
\xymatrix{
 & 0^3 \ar[d] & 0^2 \ar[d] & & \\
Z_0^3 \ar[ru] \ar[d] & Z_0^2 \ar[ru] \ar[d] & Z_0^1 \ar[r] \ar[d] & 0^1 \ar[d]  \\
 X_3 \ar[r] \ar[d] \ar[ru]& X_2 \ar[ru] \ar[r] \ar[d] & X_1 \ar[d] \ar[r] & X_0  \\
0_3 \ar[r] & C_2^3 \ar[ru] \ar[d] & C_1^3 \ar[ru] \ar[d] &  & \\
& 0_2 \ar[ru] & 0_1 &   \\
}
\]
Note that, by Lemma \ref{sfin}, the subcategory $0_1 \leftarrow C_1^n \to X_0$ is final, so the colimit of this entire diagram is $C_0^n$, as we might hope! 

Now, for $1 \leq j \leq n-1,$ note that the pushout square defining $C_j^n$ is final in $\cat D'_{/C_j^n}$ by Lemma \ref{sfin}, so we can remove the $C_j^n$ (in ascending order of $j$) by Lemma \ref{colim}. This leaves us with a diagram $\cat B$ containing no $C_j^n$'s, and with maps $0_i \to 0_j$ for $i > j$ and $0_i \to X_k$ for $i > k + 1.$  In the case $n = 3$, the diagram looks like this: 
\[
\xymatrix{
 & 0^3 \ar[d] & 0^2 \ar[d] & & \\
Z_0^3 \ar[ru] \ar[d] & Z_0^2 \ar[ru] \ar[d] & Z_0^1 \ar[r] \ar[d] & 0^1 \ar[d]  \\
 X_3 \ar[r] \ar[d] \ar[ru]& X_2 \ar[ru] \ar[r] \ar[dd] & X_1 \ar[dd] \ar[r] & X_0  \\
0_3 \ar[rd] \ar[rru] &   &  &  & \\
& 0_2 \ar[rruu] \ar[r] & 0_1  &   \\
}
\]
Now for $2 \leq k \leq n$, we will lift the maps $0_k \to X_{k-2}$ to maps $0_k \to 0^{k-1}.$ This can of course be done without changing the commutativity of the diagram. To check that the new diagram has the same colimit as the old, we need to check that, for arbitrary $d \in \cat B$ mapping to $0_k$, $0^{k-1}$ and $d$ have a coproduct in $\cat D$. We check the three possible cases:
\begin{itemize}
    \item If $d = 0_k,$ then $0^{k-1} \coprod d$ is $X_{k-1}.$
    \item If $d = 0_{k+1}$ or $d = X_k$ then $0^{k-1} \coprod d$ is $Z_0^{k-1}.$
    \item Otherwise, $0^{k-1} \coprod d = 0^{k-1}$
\end{itemize}
We conclude by Lemma \ref{lift} with $(a,b,c) = (0_k, 0^{k-1}, X_{k-2})$ that the modified diagram, which we will call $W$, still has the same colimit. In the case $n  = 3$, we have the following diagram:
\[
\xymatrix{
 & 0^3 \ar[d] & 0^2 \ar[d] & & \\
Z_0^3 \ar[ru] \ar[d] & Z_0^2 \ar[ru] \ar[d] & Z_0^1 \ar[r] \ar[d] & 0^1 \ar[d]  \\
 X_3 \ar[r] \ar[d] \ar[ru]& X_2 \ar[ru] \ar[r] \ar[dd] & X_1 \ar[dd] \ar[r] & X_0  \\
0_3 \ar[dr] \ar[rruuu] &  &  &  & \\
& 0_2 \ar[rruuu] \ar[r] & 0_1  &   \\
}
\]
Now, we apply Lemmas \ref{colim} and \ref{sfin} again to remove $X_0$, at which point $X_1$ only has one arrow pointing out of it and we can remove it by Lemma \ref{qa}. But now Lemmas \ref{colim} and \ref{sfin} tell us we can remove $Z_0^1$, and Lemma $\ref{qa}$ lets us remove $X_2.$ We can continue until we've removed all the $X_i$ and all the $Z_0^i$ except for $Z_0^n.$ 

We're left with a diagram containing only $Z_0^n$, $0^i$ and $0_i$, with $Z_0^n$ initial and maps $0^i \to 0_j$, $0_i \to 0_j$, $0^i \to 0^j$, and $0_i \to 0^j$ whenever $i > j.$ 
\[
\xymatrix{
 & 0^3 \ar[dddd] \ar[r] & 0^2 \ar[dddd] \ar[dr] & & \\
Z_0^3 \ar[ru] \ar[dd] & & & 0^1  \\
 &  &  &  \\
0_3 \ar[dr] \ar[rruuu] &  &  &  & \\
& 0_2 \ar[rruuu] \ar[r] & 0_1  &   \\
}
\]
Applying \cite{htt} 4.4.2.2 repeatedly shows that the colimit of this diagram is $\Sigma^n Z_0^n,$ as desired. 
\end{proof}

\bibliographystyle{unsrt}
\bibliography{main.bbl}

\begin{thebibliography}{10}

\bibitem{isaksen2023stable}
Daniel~C Isaksen, Guozhen Wang, and Zhouli Xu.
\newblock Stable homotopy groups of spheres: From dimension 0 to 90.
\newblock {\em Publications math{\'e}matiques de l'IH{\'E}S}, 137(1):107--243,
  2023.

\bibitem{hill2016nonexistence}
Michael~A Hill, Michael~J Hopkins, and Douglas~C Ravenel.
\newblock On the nonexistence of elements of kervaire invariant one.
\newblock {\em Annals of Mathematics}, pages 1--262, 2016.

\bibitem{browder1969kervaire}
William Browder.
\newblock The kervaire invariant of framed manifolds and its generalization.
\newblock {\em Annals of Mathematics}, pages 157--186, 1969.

\bibitem{bruner1993ext}
Robert~R Bruner.
\newblock Ext in the nineties.
\newblock {\em Contemporary Mathematics}, 146:71--71, 1993.

\bibitem{nassau2010secondary}
Christian Nassau.
\newblock On the secondary steenrod algebra.
\newblock {\em arXiv preprint arXiv:1011.2471}, 2010.

\bibitem{chua2022}
Dexter~Edralin Chua.
\newblock {\em The E3 Page of the Adams Spectral Sequence}.
\newblock PhD thesis, Harvard University, 2022.

\bibitem{wang2017triviality}
Guozhen Wang and Zhouli Xu.
\newblock The triviality of the 61-stem in the stable homotopy groups of
  spheres.
\newblock {\em Annals of Mathematics}, 186(2):501--580, 2017.

\bibitem{wang1967cohomology}
John~SP Wang.
\newblock On the cohomology of the mod-2 steenrod algebra and the non-existence
  of elements of hopf invariant one.
\newblock {\em Illinois Journal of Mathematics}, 11(3):480--490, 1967.

\bibitem{moss}
R~Michael~F Moss.
\newblock Secondary compositions and the adams spectral sequence.
\newblock {\em Mathematische Zeitschrift}, 115(4):283--310, 1970.

\bibitem{mmp}
J~Peter May.
\newblock Matric massey products.
\newblock {\em Journal of Algebra}, 12(4):533--568, 1969.

\bibitem{kochman1978chain}
Stanley~O Kochman.
\newblock A chain functor for bordism.
\newblock {\em Transactions of the American Mathematical Society},
  239:167--196, 1978.

\bibitem{gm}
Victor~KAM Gugenheim and J~Peter May.
\newblock {\em On the theory and applications of differential torsion
  products}, volume 142.
\newblock American Mathematical Soc., 1974.

\bibitem{bruner2009adams}
Robert~R Bruner.
\newblock An adams spectral sequence primer.
\newblock {\em Department of Mathematics. Wayne State University. Detroit MI},
  pages 48202--3489, 2009.

\bibitem{cohen}
Joel~M Cohen.
\newblock The decomposition of stable homotopy.
\newblock {\em Annals of Mathematics}, pages 305--320, 1968.

\bibitem{gheorghe2021special}
Bogdan Gheorghe, Guozhen Wang, and Zhouli Xu.
\newblock The special fiber of the motivic deformation of the stable homotopy
  category is algebraic.
\newblock {\em Acta Mathematica}, 226(2):319--407, 2021.

\bibitem{pstrkagowski2023synthetic}
Piotr Pstragowski.
\newblock Synthetic spectra and the cellular motivic category.
\newblock {\em Inventiones mathematicae}, 232(2):553--681, 2023.

\bibitem{dugger2010motivic}
Daniel Dugger and Daniel~C Isaksen.
\newblock The motivic adams spectral sequence.
\newblock {\em Geometry \& Topology}, 14(2):967--1014, 2010.

\bibitem{burklund2019boundaries}
Robert Burklund, Jeremy Hahn, and Andrew Senger.
\newblock On the boundaries of highly connected, almost closed manifolds.
\newblock {\em arXiv preprint arXiv:1910.14116}, 2019.

\bibitem{patchkoria2021adams}
Irakli Patchkoria and Piotr Pstragowski.
\newblock Adams spectral sequences and franke's algebraicity conjecture.
\newblock {\em arXiv preprint arXiv:2110.03669}, 2021.

\bibitem{burklund2021extension}
Robert Burklund.
\newblock An extension in the adams spectral sequence in dimension 54.
\newblock {\em Bulletin of the London Mathematical Society}, 53(2):404--407,
  2021.

\bibitem{burklundsynthetic}
Robert Burklund.
\newblock Synthetic cookware v2.0.
\newblock {\em preparation, draft available on the author’s webpage}.

\bibitem{mg}
Aaron Mazel-Gee.
\newblock {\em Goerss-Hopkins obstruction theory via model infinity
  categories}.
\newblock University of California, Berkeley, 2016.

\bibitem{boardman1973homotopy}
John~Michael Boardman and Rainer~M Vogt.
\newblock {\em Homotopy invariant algebraic structures on topological spaces},
  volume 347.
\newblock Springer, 1973.

\bibitem{joyal2002quasi}
Andr{\'e} Joyal.
\newblock Quasi-categories and kan complexes.
\newblock {\em Journal of Pure and Applied Algebra}, 175(1-3):207--222, 2002.

\bibitem{joyal2008notes}
Andr{\'e} Joyal.
\newblock Notes on quasi-categories.
\newblock {\em preprint}, 2008.

\bibitem{htt}
Jacob Lurie.
\newblock {\em Higher topos theory}.
\newblock Princeton University Press, 2009.

\bibitem{ha}
Jacob Lurie.
\newblock Higher algebra. 2021.
\newblock {\em Preprint, available at http://www. math. harvard. edu/\~{}
  lurie}, 2021.

\bibitem{glasman2013day}
Saul Glasman.
\newblock Day convolution for infinity-categories.
\newblock {\em arXiv preprint arXiv:1308.4940}, 2013.

\bibitem{gheorghe2022c}
Bogdan Gheorghe, Daniel~C Isaksen, Achim Krause, and Nicolas Ricka.
\newblock C-motivic modular forms.
\newblock {\em Journal of the European Mathematical Society},
  24(10):3597--3628, 2022.

\bibitem{lurie2015rotation}
Jacob Lurie.
\newblock Rotation invariance in algebraic k-theory.
\newblock {\em preprint}, 2015.

\bibitem{cis}
Denis-Charles Cisinski.
\newblock {\em Higher categories and homotopical algebra}, volume 180.
\newblock Cambridge University Press, 2019.

\bibitem{cirici2020model}
Joana Cirici, Daniela~Egas Santander, Muriel Livernet, and Sarah Whitehouse.
\newblock Model category structures and spectral sequences.
\newblock {\em Proceedings of the Royal Society of Edinburgh Section A:
  Mathematics}, 150(6):2815--2848, 2020.

\bibitem{brotherston2024monoidal}
James~A Brotherston.
\newblock Monoidal model structures on filtered chain complexes relating to
  spectral sequences.
\newblock {\em arXiv preprint arXiv:2402.09207}, 2024.

\bibitem{brotherston2024distributive}
James~A Brotherston.
\newblock A distributive lattice of model structures relating to spectral
  sequences.
\newblock {\em arXiv preprint arXiv:2402.09893}, 2024.

\bibitem{ariotta2021coherent}
Stefano Ariotta.
\newblock Coherent cochain complexes and beilinson t-structures, with an
  appendix by achim krause.
\newblock {\em arXiv preprint arXiv:2109.01017}, 2021.

\bibitem{hovey2007model}
Mark Hovey.
\newblock {\em Model categories}.
\newblock Number~63. American Mathematical Soc., 2007.

\bibitem{douady11suite}
A~Douady.
\newblock La suite spectrale d’adams: structure multiplicative, exp. 19.
\newblock {\em S{\'e}minaire Henri Cartan, Ecole Normale Sup{\'e}rieure, vol},
  11(2):1958--1959.

\bibitem{hedenlund2021multiplicative}
Alice~Petronella Hedenlund.
\newblock Multiplicative tate spectral sequences.
\newblock 2021.

\bibitem{helle2017pairings}
Gard~Olav Helle.
\newblock Pairings and convergence of spectral sequences.
\newblock Master's thesis, 2017.

\bibitem{sjodin1976set}
Gunnar Sj{\"o}din.
\newblock A set of generators for ext r (k, k).
\newblock {\em Mathematica Scandinavica}, 38(2):199--210, 1976.

\bibitem{ekmm}
Anthony~D Elmendorf, I~Kriz, MA~Mandell, and JP~May.
\newblock Rings, modules, and algebras in stable homotopy theory.
\newblock In {\em American Mathematical Society Surveys and Monographs,
  American Mathematical Society}, 1995.

\bibitem{barkan2023chromatic}
Shaul Barkan.
\newblock Chromatic homotopy is monoidally algebraic at large primes.
\newblock {\em arXiv preprint arXiv:2304.14457}, 2023.

\bibitem{burklund2020galois}
Robert Burklund, Jeremy Hahn, and Andrew Senger.
\newblock Galois reconstruction of artin--tate r-motivic spectra.
\newblock {\em arXiv preprint arXiv:2010.10325}, 4, 2020.

\bibitem{keller1988aisles}
Bernhard Keller and Dieter Vossieck.
\newblock Aisles in derived categories.
\newblock {\em Bull. Soc. Math. Belg. S{\'e}r. A}, 40(2):239--253, 1988.

\bibitem{tarrio2003construction}
Leovigildo~Alonso Tarr{\'\i}o, Ana~Jerem{\'\i}as L{\'o}pez, and
  Mar{\'\i}a~Jos{\'e} Salorio.
\newblock Construction of t-structures and equivalences of derived categories.
\newblock {\em Transactions of the American Mathematical Society},
  355(6):2523--2543, 2003.

\bibitem{pstrkagowski2022abstract}
Piotr Pstragowski and Paul VanKoughnett.
\newblock Abstract goerss-hopkins theory.
\newblock {\em Advances in Mathematics}, 395:108098, 2022.

\bibitem{lawrence1969matric}
Albert~F Lawrence.
\newblock {\em Matric Massey products and matric Toda brackets in the Adams
  spectral sequence}.
\newblock PhD thesis, The University of Chicago, 1969.

\bibitem{belmont2021toda}
Eva Belmont and Hana~Jia Kong.
\newblock A toda bracket convergence theorem for multiplicative spectral
  sequences.
\newblock {\em arXiv preprint arXiv:2112.08689}, 2021.

\bibitem{wall1960generators}
CTC Wall.
\newblock Generators and relations for the steenrod algebra.
\newblock {\em Annals of Mathematics}, pages 429--444, 1960.

\bibitem{balderrama2021deformations}
William Balderrama.
\newblock Deformations of homotopy theories via algebraic theories.
\newblock {\em arXiv preprint arXiv:2108.06801}, 2021.

\end{thebibliography}


\end{document}